\documentclass[a4paper, oneside, english,reqno]{amsart}

\usepackage[utf8]{inputenx}
\usepackage{amsthm, amsmath, amssymb, amsfonts, mathrsfs, mathtools, stmaryrd}
\usepackage{babel, textcomp, url, enumerate, latexsym, graphicx, varioref, hyperref, multirow, layout, siunitx, booktabs}
\usepackage{pdflscape}
\usepackage{verbatim}
\usepackage{geometry}
\usepackage{dsfont}
\usepackage{thmtools}

\usepackage{tikz, babel, rotating, tikz-cd, pigpen}
\usetikzlibrary{matrix, arrows, decorations.pathmorphing}
\usepackage{cleveref}
\RequirePackage[color       = blue!20,
                bordercolor = black,
                textsize    = tiny,
                figwidth    = 0.99\linewidth]{todonotes}

\usepackage[all]{xy}
\usepackage{microtype}

\usepackage[backend = biber, style = alphabetic, url = false, isbn = false, doi = false, maxbibnames = 5]{biblatex}
\numberwithin{equation}{section}

\DeclarePairedDelimiter{\p}{\lparen}{\rparen}

\DeclarePairedDelimiter{\ip}{\langle}{\rangle}

\makeatletter
\let\c@equation\c@subsubsection

\makeatother
\newtheorem{corollary}[subsubsection]{Corollary}
\newtheorem{lemma}[subsubsection]{Lemma}
\newtheorem{prop}[subsubsection]{Proposition}

\newtheorem{theorem}[subsubsection]{Theorem}
\theoremstyle{definition}
\newtheorem{definition}[subsubsection]{Definition}
\newtheorem{construction}[subsubsection]{Construction}
\newtheorem{example}[subsubsection]{Example}

\newtheorem*{claim*}{Claim}
\theoremstyle{remark}
\newtheorem{remark}[subsubsection]{Remark}

\newcommand{\wt}{\widetilde}
\newcommand{\wh}{\widehat}
\newcommand{\ol}{\overline}

\newcommand{\defeq}{\vcentcolon=}
\newcommand{\eqdef}{=\vcentcolon}

\newcommand{\Nis}{\mathrm{Nis}}

\newcommand{\op}{\mathrm{op}}

\newcommand{\fr}{\mathrm{fr}}

\newcommand{\sspt}{\mathds{1}}
\newcommand{\A}{\mathbb{A}}

\newcommand{\D}{\mathbf{D}}

\newcommand{\G}{\mathbb{G}}
\newcommand{\HH}{\mathbb{H}}

\newcommand{\LL}{\mathbb{L}}

\newcommand{\PP}{\mathbb{P}}
\newcommand{\Q}{\mathbb{Q}}

\newcommand{\Z}{\mathbb{Z}}

\newcommand{\bfI}{\mathbf{I}}

\newcommand{\bfK}{\mathbf{K}}

\newcommand{\scrC}{\mathscr{C}}

\newcommand{\scrF}{\mathscr{F}}

\newcommand{\scrL}{\mathscr{L}}
\newcommand{\scrM}{\mathscr{M}}

\newcommand{\calA}{\mathcal{A}}

\newcommand{\calC}{\mathcal{C}}

\newcommand{\calE}{\mathcal{E}}

\newcommand{\calO}{\mathcal{O}}

\newcommand{\calU}{\mathcal{U}}
\newcommand{\calV}{\mathcal{V}}

\newcommand{\calZ}{\mathcal{Z}}

\newcommand{\rmC}{\mathrm{C}}
\newcommand{\rmc}{\mathrm{c}}

\newcommand{\rmF}{\mathrm{F}}

\newcommand{\rmH}{\mathrm{H}}
\newcommand{\rmh}{\mathrm{h}}

\newcommand{\rmK}{\mathrm{K}}
\newcommand{\rmL}{\mathrm{L}}
\newcommand{\rmM}{\mathrm{M}}

\newcommand{\rmT}{\mathrm{T}}

\newcommand{\rmW}{\mathrm{W}}

\newcommand{\DM}{\mathbf{DM}}

\newcommand{\Ab}{\mathrm{Ab}}

\newcommand{\Sch}{\mathrm{Sch}}

\newcommand{\Sm}{\mathrm{Sm}}
\newcommand{\EssSm}{\mathrm{EssSm}}
\newcommand{\SmOp}{\mathrm{SmOp}}

\newcommand{\SH}{\mathbf{SH}}

\newcommand{\Spc}{\mathrm{Spc}}

\newcommand{\Shv}{\mathrm{Shv}}

\newcommand{\eff}{\mathrm{eff}}

\newcommand{\PSh}{\mathrm{PSh}}

\newcommand{\ext}{\mathrm{ext}}

\newcommand{\MSL}{\mathrm{MSL}}
\newcommand{\red}{\mathrm{red}}

\newcommand{\pair}{\mathrm{pair}}

\newcommand{\id}{\mathrm{id}}
\newcommand{\pt}{\mathrm{pt}}

\newcommand{\MW}{\mathrm{MW}}
\newcommand{\Corr}{\mathbf{Corr}}

\DeclareMathOperator{\End}{End}

\DeclareMathOperator{\Char}{char}

\DeclareMathOperator{\im}{im}

\DeclareMathOperator{\Ext}{Ext}

\DeclareMathOperator{\Hom}{Hom}
\DeclareMathOperator{\Fun}{Fun}

\DeclareMathOperator{\coker}{coker}
\DeclareMathOperator{\Spec}{Spec}
\DeclareMathOperator{\pr}{pr}

\DeclareMathOperator{\Div}{div}

\DeclareMathOperator{\Map}{Map}

\DeclareMathOperator{\Cor}{Cor}

\DeclareMathOperator{\Mod}{Mod}

\DeclareMathOperator{\res}{res}

\DeclareMathOperator{\Fr}{Fr}
\DeclareMathOperator{\CH}{CH}
\DeclareMathOperator{\rk}{rk}
\DeclareMathOperator{\cofib}{cofib}

\newcommand{\SheafHom}{\mathscr{H}\kern-3pt om}
\SelectTips{cm}{10}

\address{Andrei Druzhinin, Chebyshev Laboratory, St. Petersburg State University, 14th Line V.O., 29B, Saint Petersburg 199178 Russia}
\email{\href{mailto:andrei.druzh@gmail.com}{andrei.druzh@gmail.com}}

\address{H{\aa}kon Kolderup, University of Oslo, Postboks 1053, Blindern, 0316 Oslo, Norway}
\email{\href{mailto:haakoak@math.uio.no}{haakoak@math.uio.no}}

\begin{document}

\title{Cohomological correspondence categories}
\author{Andrei Druzhinin and Håkon Kolderup}

\subjclass[2010]{14F05, 14F35, 14F42, 19E15}
\keywords{Correspondences, motives, motivic homotopy theory}

\maketitle
\begin{abstract}
We prove that homotopy invariance and cancellation properties are satisfied by any category of correspondences that is defined, via Calmès and Fasel's construction, by an underlying cohomology theory. In particular, this includes any category of correspondences arising from the cohomology theory defined by an $\MSL$-algebra.
\end{abstract}

\addtocontents{toc}{\protect\setcounter{tocdepth}{1}}
\section{Introduction}
Originally envisioned by Grothendieck, the theory of motives was set in new light by Beilinson's conjecture on the existence of certain \emph{motivic complexes}, from which it should be possible to derive a satisfactory motivic cohomology theory. This point of view ultimately led to Suslin and Voevodsky's construction of the derived category of motives $\DM(k)$ over any field $k$ \cite{Voe-motives}. The basic ingredient of this construction is the category $\Cor_k$ of finite correspondences over $k$. Finite correspondences define an additive category, and presheaves on this category---baptized \emph{presheaves with transfers}---are exceptionally well behaved. Indeed, presheaves with transfers carry a very rich theory, satisfying fundamental properties such as preservation of homotopy invariance under sheafification \cite{Voe-hty-inv}, and a cancellation property with respect to smashing with $\G_m$ \cite{Voe-cancel}. These results are crucial in order to obtain a good category of motivic complexes.

Shortly after Suslin and Voevodsky's introduction of motivic complexes, a ``nonlinear'' version of $\DM(k)$ was defined by Morel and Voevodsky \cite{Morel-Voevodsky} in the context of motivic homotopy theory. In this more general setting, the motivic stable homotopy category $\SH(k)$ was constructed, most notably via the $\A^1$-localization and the $\PP^1$-stabilization process. The category $\SH(k)$ is equipped with an adjunction 
\begin{align}
\gamma^*:\SH(k)\rightleftarrows\DM(k):\gamma_*\label{eq:adjunction}
\end{align} 
such that the image of the unit for the symmetric monoidal structure on $\DM(k)$ is mapped to the motivic Eilenberg--Mac Lane spectrum $\rmH\Z$ in $\SH(k)$ under $\gamma_*$. In fact, this adjunction exhibits $\DM(k)$ as the category of modules over the ring spectrum $\rmH\Z$ (at least after inverting the exponential characteristic of $k$) \cite{MZmod}. Furthermore, the restriction of $\gamma_*$ to the heart of the homotopy t-structure on $\DM(k)$ is fully faithful. In fact, with rational coefficients, the category $\SH(k)_\Q$ splits into a plus part and a minus part, where the plus part is equivalent to $\DM(k,\Q)$ \cite{Cisinski-Deglise}. Informally we can think of $\DM(k,\Q)$ as consisting of the oriented part of $\SH(k)_\Q$. 

Several alternative and refined versions of the category of correspondences have been introduced in the wake of Suslin and Voevodsky's pioneering work, many of which attempt to provide a better approximation to the motivic stable homotopy category than $\DM(k)$. In particular, it is desirable to construct correspondences that capture also the unoriented information contained in $\SH(k)$. Examples include
\begin{itemize}
\item the category $\Z\rmF_*(k)$ of linear framed correspondences, introduced by Voevodsky and further developed by Garkusha and Panin \cite{Framed};
\item $\rmK_0^\oplus$-, and $\rmK_0$-correspondences, studied by Suslin and Walker in \cite{Suslin-Grayson, Walker};
\item the category $\wt\Cor_k$ of finite Milnor--Witt correspondences, introduced by Calmès--Déglise--Fasel \cite{Calmes-Fasel, MW-cplx}; and
\item the category $\mathrm{GW}\mkern-3mu\Cor_k$ of finite Grothendieck--Witt correspondences defined by the first author in \cite{DruDMGWeff}.
\end{itemize}
To exemplify to what extent the above categories succeed in providing better approximations to $\SH(k)$, let us mention that framed correspondences classify infinite $\PP^1$-loop spaces \cite{five-authors}, and the heart of the category $\wt\DM(k)$ associated to $\wt\Cor_k$ is equivalent to the heart of $\SH(k)$ (with respect to the homotopy t-structure) \cite{framed-MW}.

Along with the introduction of each new category of correspondences follows the need to prove fundamental properties like strict homotopy invariance and cancellation in order to produce a satisfactory associated derived category of motives. For the above examples, this is achieved in \cite{framed-cancel,hty-inv, Suslin-Grayson, MW-cancel, MW-cplx,GWStrHomInv, GW-cancel}. The aim of this note is to establish these properties simultaneously for a certain class of correspondence categories, namely those that are defined by an underlying cohomology theory (see \Cref{def:ACor} for the precise meaning). This includes Voevodsky's finite correspondences---which can be defined using the cohomology theory $\CH^*$ of Chow groups---as well as finite Milnor--Witt correspondences $\wt\Cor_k$, which are defined using Chow--Witt groups $\wt\CH^*$. More generally, any ring spectrum $E\in\SH(k)$ that is an algebra over Panin and Walter's algebraic cobordism spectrum $\MSL$ \cite{Panin-Walter} gives rise to a cohomological correspondence category.

\subsection{Outline}
In \Cref{section:coh-thy} we introduce the axioms for a cohomology theory $A^*$ needed to build the associated category $\Cor^A_k$ of finite $A$-correspondences. The definition of the category $\Cor^A_k$ is given in \Cref{section:ccorrs}. In addition we give in \Cref{section:ccorrs} a number of constructions in the category $\Cor^A_k$. Most notably, \Cref{constr:lrangleclass} ensures that a regular function on a smooth relative curve along with a trivialization of the relative canonical class gives rise to a finite $A$-correspondence; this construction is used to define all the finite $A$-correspondences needed to prove strict homotopy invariance and cancellation.

\Cref{section:framed} is a brief comparison between our construction of $A$-correspondences and framed correspondences. This is done by constructing a functor from the category of framed correspondences $\Fr_*(k)$ to $\Cor^A_k$.

Sections \ref{section:inj-aff-line}, \ref{section:rel-aff-line}, \ref{section:inj-loc-sch} and \ref{section:et-exc} are devoted to the proof of the strict homotopy invariance property of homotopy invariant presheaves on $\Cor^A_k$. The proof breaks down into several excision results as well as a moving lemma, each of which is treated in its own section.

In \Cref{section:cancel} we show the cancellation theorem for finite $A$-correspondences, following the technique in Voevodsky's original proof \cite{Voe-cancel}.

Finally, in \Cref{section:motives} we use the previous results to establish a well behaved category of motivic complexes $\DM_A(k)$ associated to the category $\Cor^A_k$, and we show several properties expected of this category. In particular, we define $A$-motivic cohomology in this category, and show that $\DM_A(k)$ comes equipped with an adjunction to $\SH(k)$ parallelling \eqref{eq:adjunction}. Note that these constructions are for the most part standard. For this reason we keep it rather brief on certain formal aspects of the constructions, and refer the interested reader to, e.g., \cite{Voe-motives, MVW} or \cite{MW-cplx} for further details.

\Cref{appendix} is a collection of the geometric results used in the proofs of the excision theorems. 

\subsection{Relationship to other works}\label{section:rel-to-other-works}
In the independent project \cite{five-authors2}, the construction of the category $\Cor^E_k$ of \Cref{section:SL-orient} is generalized to arbitrary ring spectra in $\SH(S)$ over a base scheme $S$. Let us also mention that functors from the category of framed correspondences to other correspondence categories have been considered by several authors. The original construction of a functor $\Fr_*(k)\to\wt\Cor_k$ from framed correspondences to finite Milnor--Witt correspondences was given by Déglise and Fasel in \cite{MW-cplx}. In \cite[§4.2]{five-authors2}, the functor of Déglise and Fasel was refined to an $\rmh\Spc$-enriched functor $\Phi^E\colon\rmh\Corr^\fr(\Sch_S)\to\rmh\Corr^E(\Sch_S)$ from the homotopy category of the $\infty$-category of framed correspondences to finite $E$-correspondences.

\subsection{Conventions and notation}Throughout, the symbol $k$ will denote a field, and the symbol $\G_m\defeq \Spec(k[t,t^{-1}])$ will denote the multiplicative group scheme over $k$. In certain sections we will also need to put some restrictions on the field $k$; this will be stated in the beginning of the relevant section.

By a \emph{base scheme} we mean a noetherian scheme of finite Krull dimension. If $S$ is a base scheme, we let $\Sm_S$ denote the category of schemes that are smooth, separated and of finite type over $S$. 
By an \emph{essentially smooth scheme} we mean a scheme that is a projective limit of open immersions of smooth ones. We denote the category of essentially smooth schemes by $\EssSm_S$. 
If $f\colon X\to Y$ is a morphism in $\Sm_S$ (or $\EssSm_S$), we let $\omega_f\defeq \omega_{X/S}\otimes f^*\omega_{Y/S}^{-1}$ denote the relative canonical sheaf. Moreover, we may write simply $\omega_Y$ for $\omega_{X\times_S Y/X}$. 
In the case of smooth (or essentially smooth) schemes $X,Y\in\Sm_k$ (or $\EssSm_k$) over a field $k$, we will often abbreviate $X\times_k Y$ to $X\times Y$; $\A^n_k$ to $\A^n$ and $\PP^n_k$ to $\PP^n$.
Throughout, we will let $i_0$ and $i_1$ denote the zero-, respectively the unit section $i_0,i_1\colon\Spec k\to\A^1$. If we for example need to emphasize that $\A^2$ has coordinates $(x,y)$, we may for brevity denote this by $\stackrel{(x,y)}{\A^2}$. This notation will in particular be used in the tables in Sections \ref{section:inj-aff-line}, \ref{section:rel-aff-line}, \ref{section:inj-loc-sch} and \ref{section:et-exc}. 

If $\scrL$ is a line bundle on a scheme $X$ and $s\in \Gamma(X,\scrL)$ is a section of $\scrL$, 
we will denote by $Z(s)\subseteq X$ the vanishing locus of $s$. 
We say that a section $s\in \Gamma(X,\scrL)$ is \emph{invertible} if 
the homomorphism $\calO_X\to \scrL$ defined by $s$ is an isomorphism.

We denote by $\Map_\scrC(X,Y)$ the mapping spaces of an $\infty$-category $\scrC$, and write $[X,Y]_\scrC\defeq\pi_0\Map_\scrC(X,Y)$. If $\scrC$ is any category, we denote by $\PSh(\scrC)\defeq\Fun(\scrC^\op,\Spc)$ the $\infty$-category of presheaves on $\scrC$, and for a ring $R$ we denote by $\PSh(\scrC;R)$ the $\infty$-category of presheaves of $R$ modules on $\scrC$. Moreover, we let $\PSh_\Sigma(\scrC)$ denote the full subcategory of $\PSh(\scrC)$ spanned by presheaves that carry finite coproducts to finite products \cite[§5.5.8]{Lurie}.

\subsection{Acknowledgments}We are grateful to Alexey Ananyevskiy, Frédéric Déglise, Jean Fasel, Ivan Panin, and Paul Arne Østvær for helpful discussions and comments. We would also like to thank Marc Hoyois for explaining to us how to use the six functors formalism to construct pushforwards.

Both authors gratefully acknowledge the support provided by the RCN Frontier Research Group Project no. 250399 ``Motivic Hopf equations''.
The first author would also like to thank ``Native towns'', a social investment program of PJSC ``Gazprom Neft'', for support.

Finally, we would like to thank the anonymous referee for several very helpful comments and remarks.

\addtocontents{toc}{\protect\setcounter{tocdepth}{2}}

\section{Twisted cohomology theories with support}\label{section:coh-thy}
Let $S$ be a base scheme. We denote by $\SmOp_S^\rmL$ the category of triples $(X,U,\scrL)$, where $X\in\Sm_S$ is separated, smooth and of finite type over $S$, $U$ is an open subscheme of $X$ and $\scrL$ is a line bundle on $X$. A morphism $(X,U,\scrL)\to(Y,V,\scrM)$ in $\SmOp^\rmL_S$ consists of a pair $(f,\alpha)$ of a morphism of $S$-schemes $f\colon X\to Y$ such that $f(U)\subseteq V$, and an isomorphism $\alpha\colon \scrL \xrightarrow{\cong}f^*\scrM$. Note that there is an embedding $\Sm_S\to\SmOp^\rmL_S$ given by $X\mapsto(X,\varnothing,\calO_X)$. For any $(X,U,\scrL)\in\SmOp^\rmL_S$, we will write $i_U$ for the inclusion $i_U\colon U\to X$ and $j_U$ for the inclusion $j_U\colon(X,\varnothing,\scrL)\to(X,U,\scrL)$. In the case when $U=\varnothing$, we will often denote the triple $(X,\varnothing,\scrL)\in\SmOp_S^\rmL$ simply by $(X,\scrL)$.

\begin{definition}\label{def:pre-coh}
A \emph{twisted pre-cohomology theory} is a graded functor 
\[
A^*\colon (\SmOp_S^\rmL)^\op\to\Ab^\Z
\]
which satisfies the following properties:
\begin{itemize}
\item[(a)](Localization) There is a natural transformation $\partial\colon A^*(X,U,\scrL)\to A^{*+1}(U,i_U^*\scrL)$ of degree $+1$ which fits into an exact sequence
\[
A^*(X,\scrL)\xrightarrow{i_U^*} A^*(U,i_U^*\scrL)\xrightarrow{\partial} A^{*+1}(X,U,\scrL)\xrightarrow{j_U^*} A^{*+1}(X,\scrL).
\]
\item[(b)] (Étale excision) Suppose that $f\colon X\to Y$ is an étale morphism of smooth $S$-schemes. Assume moreover that $Z\subseteq Y$ is a closed subset such that $f|_{f^{-1}(Z)}\colon f^{-1}(Z)\to Z$ is an isomorphism. Then the pullback homomorphism
\[
f^*\colon A^n(Y,Y\setminus Z,\scrL)\to A^n(X,X\setminus f^{-1}(Z),f^*\scrL)
\]
is an isomorphism for any line bundle $\scrL$ on $Y$ and any $n\in\Z$.
\end{itemize}
If $(X,U,\scrL)\in\SmOp^\rmL_S$, let $Z\defeq X\setminus U$ be the closed complement of $U$. We then write $A^*_Z(X,\scrL)\defeq A^*(X,U,\scrL)$. The map $j_U^*\colon A_Z^*(X,\scrL)\to A^*(X,\scrL)$ is called the \emph{extension of support-homomorphism}.
\end{definition}

\begin{remark}
\Cref{def:pre-coh} is but a twisted version of Panin and Smirnov's definition of a cohomology theory considered for example in \cite{Panin-Smirnov}, except that for our purposes we need not assume the axiom of homotopy invariance. In the case of oriented homotopy invariant theories, our definition coincides with Panin and Smirnov's definition.
\end{remark}

\begin{remark}\label{rem:Zarex}
The axiom of étale excision in \Cref{def:pre-coh} implies that there is a canonical isomorphism $A^*_{Z_1\amalg Z_2}(X,\scrL)\cong A^*_{Z_1}(X,\scrL)\oplus A^*_{Z_2}(X,\scrL)$, i.e., that the cohomology theory $A^*$ also satisfies Zariski excision. In fact, Zariski excision is enough to prove most of the results below. The only places where we need étale excision are in the construction of the functor from framed correspondences to $A$-correspondences in \Cref{section:framed}, and in the proof that $A$-transfers are preserved under Nisnevich sheafification (\Cref{thm:ext}). Furthermore, the latter case only requires étale excision on local schemes. In \Cref{cor:hty-inv-exc} we show that a \emph{homotopy invariant} cohomology theory satisfying Zariski excision will automatically satisfy étale excision on local schemes.
\end{remark}

\begin{definition}\label{def:coh}
Let $A^*$ be a twisted pre-cohomology theory. 
Suppose that we in addition are given the following data:
\begin{enumerate} 
\item (Pushforward) 
For any morphism $f\colon X\to Y\in \Sm_S$ of smooth equidimensional $S$-schemes of constant relative dimension $d$, and any closed subset $Z\subseteq X$ such that $f|_Z$ is finite, we have a \emph{pushforward} 
homomorphism 
\[
f_*\colon A^n_Z(X,\omega_f\otimes f^*\scrL)\to A^{n-d}_{f(Z)}(Y,\scrL)
\]
for any $n\ge0$ and any line bundle $\scrL$ on $Y$. 
\item (External product)
The cohomology theory is a \emph{ring cohomology theory}, i.e., there is an associative product structure
\[
\times\colon A^n_{Z_1}(X,\scrL)\otimes A_{Z_2}^m(Y,\scrM)\to A_{Z_1\times_S Z_2}^{n+m}(X\times_S Y,\scrL\boxtimes \scrM)
\]
and a unit $1\in A^0(S)$.
\end{enumerate}
We say that a pre-cohomology theory $A^*$ equipped with the homomorphisms $f_*$ and the product $\times$ as above forms a \emph{good cohomology theory} if the following properties hold:
\begin{enumerate}
\setcounter{enumi}{2}
\item (Pushforward functoriality) 
The homomorphisms $f_*$ are functorial in the sense that $\id_*=\id$, and if 
$
(X_1,U_1,\scrL_1)\xrightarrow{f}(X_2,U_2,\scrL_2)\xrightarrow{g}(X_3,U_3,\scrL_3)
$
are composable morphisms in $\SmOp_S^\rmL$ finite on the supports $Z_i\defeq X_i\setminus U_i$, then the diagram
\[\begin{tikzcd}
A^{n-d_f}_{f(Z_1)}(X_2,\omega_g\otimes g^*\scrL_3)\ar{r}{g_*} & A^{n-d_{gf}}_{gf(Z_1)}(X_3,\scrL_3)\\
A^{n}_{Z_1}(X_1,\omega_f\otimes f^*\scrL_2)\ar{u}{f_*}\ar{ur}[swap]{(gf)_*} &
\end{tikzcd}\]
is commutative. Here $d_f$, $d_g$ and $d_{gf}$ are the respective relative dimensions of the morphisms.

\item (External product functoriality)
The external product $\times$ commutes with pullbacks in the sense that if $f\colon(X,f^*\scrL)\to(Y,\scrL)$ and $g\colon(X',g^*\scrL')\to(Y',\scrL')$ are morphisms in $\SmOp_S^\rmL$, then the diagram
\[\begin{tikzcd}
A^n(Y,\scrL)\otimes A^m(Y',\scrL')\ar{r}{\times}\ar{d}[swap]{f^*\otimes g^*}\ & A^{n+m}(Y\times_S Y',\scrL\boxtimes\scrL')\ar{d}{(f\times g)^*}\\
A^n(X,f^*\scrL)\otimes A^m(X',g^*\scrL')\ar{r}{\times} & A^{n+m}(X\times_S X',f^*\scrL\boxtimes g^*\scrL')
\end{tikzcd}\]
is commutative.

\item (Base change) 
For any strongly transversal square (defined in \Cref{def:transsq}) that is equipped with a set of compatible line bundles (defined in \Cref{def:base-change}) the diagram
\[\begin{tikzcd}
A^n_{\phi_Y^{-1}(Z)}(Y',\scrM')\ar{r}{i'_*} & A_{i'(\phi_Y^{-1}(Z))}^{n-d'}(X',\scrL')\\
A^n_Z(Y,\scrM)\ar{u}{\phi^*_Y}\ar{r}{i_*} & A^{n-d}_{i(Z)}(X,\scrL)\ar{u}[swap]{\phi_X^*},
\end{tikzcd}\]
is commutative.
\item (Projection formula) Suppose that $f\colon X\to Y$ and $Z\subseteq X$ satisfy the hypotheses of (1), and let $W\subseteq Y$ be a closed subset. Let moreover $\scrL$ and $\scrM$ be two line bundles on $Y$. Given any two cohomology classes $\alpha\in A^n_Z(X,\omega_f\otimes f^*\scrL)$ and $\beta\in A^m_{W}(Y,\scrM)$, we then have
\[
f_*(\alpha)\smile\beta=f_*(\alpha\smile f^*\beta).
\]
\item (Graded commutativity)
For any
$\alpha\in A^n_Z(X,\scrL)$ and $\beta\in A^m_Z(X,\scrL)$, 
we have
\[
\alpha\smile \beta = \ip{-1}^{nm}(\beta\smile\alpha).
\]
Here $\ip{-1}\in A^0(S)$ is given as the pushforward $\ip{-1}\defeq(\id_S,-1)_*(1)$; see \Cref{def:ip}. Hence the ring $A^*(S)$ is $\ip{-1}$-graded commutative.
\end{enumerate}
\end{definition}

\begin{remark}\label{rmk:cup}
The existence of an external product $\times$ as in \Cref{def:coh} (2) is equivalent to the existence of a cup product 
$
\smile\colon A^n_{Z_1}(X,\scrL)\otimes A^m_{Z_2}(X,\scrM)\to A^{n+m}_{Z_1\cap Z_2}(X,\scrL\otimes\scrM)$;
see \cite[Definition 1.5]{Panin-Smirnov} for further details on this.
\end{remark}

\begin{definition}\label{def:transsq}
Let
\begin{equation}\label{diag:transsq}
\begin{tikzcd}
Y'\ar{r}{i'}\ar{d}[swap]{\phi_Y} & X'\ar{d}{\phi_X} \\
Y\ar{r}{i} & X
\end{tikzcd}
\end{equation}
be a Cartesian square of smooth $S$-schemes.
The square \eqref{diag:transsq} is called \emph{transversal} if 
the corresponding sequence 
\[
0\to g^*(\Omega_X)\to \phi_Y^*(\Omega_Y)\oplus {i^\prime}^*(\Omega_{X^\prime})\to \Omega_{Y^\prime}\to 0
\] 
is exact,
where $g\defeq\phi_X\circ i^\prime=i\circ \phi_Y$.
Note that for any transversal square, the isomorphism $d\phi_Y$
induces an isomorphism
$d\phi_Y\colon  \phi_Y^*\omega_i\xrightarrow{\cong} \omega_{i^\prime}$.

A transversal square \eqref{diag:transsq} is called \emph{strongly transversal}
if one of the following two conditions are satisfied:
\begin{itemize}
\item The morphisms $i$ and $i'$ are closed embeddings. 
\item The morphisms $\phi_X$ and $\phi_Y$ are smooth and surjective.
\end{itemize}
\end{definition}




\begin{definition}\label{def:base-change}
Suppose that the square \eqref{diag:transsq} is strongly transversal. Then a \emph{compatible set of line bundles} on the square \eqref{diag:transsq} consists of the following data:
\begin{itemize}
\item Line bundles $\scrL,\scrL',\scrM,\scrM'$ on respectively $X,X',Y$ and $Y'$.
\item Isomorphisms of line bundles
\[\begin{array}{ll}
\alpha\colon\phi_X^*\scrL\xrightarrow{\cong}\scrL'; & \gamma\colon i^*\scrL\otimes\omega_i\xrightarrow{\cong}\scrM;\\
\beta\colon\phi_Y^*\scrM\xrightarrow{\cong}\scrM';  & \delta\colon (i')^*\scrL'\otimes\omega_{i^\prime}\xrightarrow{\cong} \scrM'.
\end{array}\]
\end{itemize}
We furthermore require that $\beta\circ\phi_Y^*(\gamma)$ corresponds to $\delta\circ ((i')^*(\alpha)\otimes\id_{\omega_{i'}})$ under the isomorphism 
\[\Hom_{\calO_{Y'}}(\phi_Y^*i^*\scrL\otimes\phi_Y^*\omega_i,\scrM')\cong \Hom_{\calO_{Y'}}((i')^*\phi_X^*\scrL\otimes\omega_{i^\prime},\scrM')\]
induced by the canonical isomorphism $\phi_Y^*\omega_i\cong \omega_{i^\prime}$ for the transversal square.
\end{definition}

\section{Cohomological correspondences}\label{section:ccorrs}

We are now ready to extend Calmès and Fasel's definition of finite Milnor--Witt correspondences \cite{Calmes-Fasel} to our setting:


\begin{definition}\label{def:ACor}
Let $S$ be a connected base scheme, and suppose that $A^*$ is a good cohomology theory on $\SmOp_S^\rmL$. Assume further that $p\colon X\to S$ is a smooth map of constant relative dimension $d$. Denote by $\calA_0(X/S)$ the set of \emph{admissible subsets\footnote{Note that for any $X,Y\in\Sm_k$ we have $\calA_0(X\times Y/X)=\calA(X,Y)$, where $\calA(X,Y)$ is the set of admissible subsets of $X\times Y$ in the sense of \cite[Definition 4.1]{Calmes-Fasel}.} of $X$ relative to $S$}---that is, closed subsets $T$ of $X$ such that each irreducible component of $T_\red$ is finite and surjective over $S$ via the morphism $p$. The set $\calA_0(X/S)$ is partially ordered by inclusions. As the empty set has no irreducible components, it is admissible. 
If $X$ is connected, we define the group of \emph{finite relative $A$-cycles on $X$} as
\[
\rmC^A_0(X/S)\defeq \varinjlim_{T\in\calA_0(X/S)} A_T^{d}(X,\omega_{X/S}).
\]
If $X$ is not connected, we may write $X=\coprod_j X_j$ where the $X_j$'s are the connected components of $X$. We then set
$
\rmC^A_0(X/S)\defeq\prod_j \rmC^A_0(X_j/S).
$

Now let $k$ be a field, and suppose further that $S\in\Sm_k$. Let $\Cor^A_S$ denote the category whose objects are the same as the objects of $\Sm_S$, i.e., smooth separated schemes of finite type over $S$, and morphisms defined as follows. Let $X,Y\in\Sm_S$, and suppose first that $X$ and $Y$ are connected. We define the group of \emph{finite relative $A$-correspondences from $X$ to $Y$ as}
\[
\Cor^A_S(X,Y)\defeq \rmC^A_0(X\times_S Y/X).\]
Note in particular that $\Cor^A_S(X,S)=A^0(X)$ for any $X\in\Sm_S$. If $X$ or $Y$ is not connected, let $X=\coprod_iX_i$ and $Y=\coprod_jY_j$ denote the connected components of $X$ and $Y$. Then we put $\Cor^A_S(X,Y)\defeq\prod_{i,j}\Cor^A_S(X_i,Y_j)$. If $S=\Spec k$, we refer to $\Cor_k^A(X,Y)$ simply as the group of \emph{finite $A$-correspondences from $X$ to $Y$}. 

Composition of finite relative $A$-correspondences is defined in an identical manner as \cite[§4.2]{Calmes-Fasel}. Indeed, if $\alpha\in\Cor^A_S(X,Y)$ and $\beta\in\Cor^A_S(Y,Z)$, we put
\begin{align}
\beta\circ\alpha\defeq(p_{XZ})_*\p*{p_{XY}^*\alpha\smile p_{YZ}^*\beta}.\label{eq:corr-comp}
\end{align}
Here we write $p_{XY}$ for the projection $p_{XY}\colon X\times_S Y\times_S Z\to X\times_S Y$, and similarly for the other two maps. An identical proof as that of \cite[Lemma 4.13]{Calmes-Fasel} then shows that the groups $\Cor^A_S(X,Y)$ form the mapping sets of a (discrete) category $\Cor^A_S$ whose objects are the same as those of $\Sm_S$. We refer to $\Cor_S^A$ as the category of \emph{finite relative $A$-correspondences}. In the case when $S=\Spec k$, we refer to $\Cor_k^A$ simply as the category of \emph{finite $A$-correspondences}.

Finally, we define the \emph{homotopy category} $\ol\Cor^A_S$ of $\Cor_S^A$ as follows. The objects of $\ol\Cor^{A}_S$ are the same as those of $\Cor^{A}_S$, and the morphisms are given by
\begin{align*}
&\ol\Cor^{A}_S(X,Y)\defeq\Cor^{A}_S(X,Y)/\sim_{\A^1}\\
&=\coker\p*{\Cor^{A}_S(\A^1_S\times_S X,Y)\xrightarrow{i_0^*-i_1^*}\Cor^{A}_S(X,Y)}.
\end{align*}
We write $[\alpha]$ for the class in $\ol\Cor_S^A$ of a finite relative $A$-correspondence $\alpha$ from $X$ to $Y$.
\end{definition}

\subsubsection{Graph functors}We define a graph functor $\gamma_{A,S}\colon\Sm_S\to\Cor^A_S$ similarly as \cite[§4.3]{Calmes-Fasel}: the functor $\gamma_{A,S}$ is the identity on objects, and if $f\colon X\to Y$ is a morphism in $\Sm_S$, we let $\gamma_{A,S}(f)\defeq i_*(1)$. Here $i\colon \Gamma_f\to X\times_S Y$ is the embedding of the graph of $f$, and 
$
i_*\colon A^0(\Gamma_f,\calO_{\Gamma_f})\to A^{\dim Y}_{\Gamma_f}(X\times_S Y,\omega_Y)
$
is the induced pushforward. If $S=\Spec k$, we will write $\gamma_A$ for the graph functor. We will often abuse notation and write simply $f$ instead of $\gamma_{A,S}(f)$.

\subsubsection{Symmetric monoidal structure}Defining $X\oplus Y\defeq X\amalg Y$ turns $\Cor^A_S$ into an additive category with zero-object the empty scheme. Moreover, $\Cor^A_S$ is symmetric monoidal, with tensor product $\otimes$ defined by $X\otimes Y\defeq X\times_S Y$ on objects, and given by the external product on morphisms.

\begin{lemma}\label{lemma:corr-cat}
The category $\Cor^A_k$ is a (discrete) correspondence category in the sense of \cite[Definition 4.1]{Classification} (see also \cite[§2]{Garkusha-reconst}).
\end{lemma}

\begin{proof}
This follows from \cite[Proposition 4.5]{Classification}.
\end{proof}

\subsubsection{}For $S$ a smooth $k$-scheme there is a functor $\ext_S\colon \Cor^A_k\to \Cor^A_S$ defined as follows. For any $X\in\Sm_k$, let $X_S\defeq X\times_k S$. Let $X,Y\in\Sm_k$; by working with one connected component at a time, we may assume that $X$ and $Y$ are connected. By the universal property of fiber products we have a morphism $f\colon X_S\times_S Y_S\to X\times Y$, which induces a pullback morphism
\[
f^*\colon A_T^{\dim Y}(X\times Y,\omega_{Y})\to A_{f^{-1}(T)}^{\dim Y}(X_S\times_S Y_S,f^*\omega_{Y})
\]for any $T\in\calA_0(X\times Y/X)$. As finiteness and surjectivity are preserved under base change we have $f^{-1}(T)\in\calA_0(X_S\times_S Y_S/X_S)$. 
Moreover, the canonical sheaf $\omega_{X/k}$ pulls back over $X_S$ to $\omega_{X_S/S}$, and similarly for $\omega_{Y/k}$. Hence $f^*\omega_{X\times Y/X}\cong\omega_{X_S\times_S Y_S/X_S}$. Since pullbacks commute with extension of support, we get an induced map on the colimit
\[
\ext_S\colon \Cor^A_k(X,Y)\to \rmC^A_0(X_S\times_S Y_S/X_S)=\Cor^A_S(X_S,Y_S).
\]
It follows from the base change axiom applied to the diagram
\[\begin{tikzcd}
X_S\times_S Y_S\times_S Z_S\ar{r}{p_{X_SY_S}}\ar{d}[swap]{f_{XYZ}} & X_S\times_S Y_S\ar{d}{f_{XY}}\\
X\times Y \times Z\ar{r}{p_{XY}} & X\times Y
\end{tikzcd}\]
that the map $\ext_S$ preserves composition of finite $A$-correspondences. Thus we obtain a functor $\ext_S\colon\Cor^A_k\to \Cor^A_S$.

\subsubsection{}In the opposite direction there is a ``forgetful'' functor $\res_S\colon\Cor^A_S\to \Cor^A_k$ induced by pushforwards.
Indeed, let $X,Y\in \Sm_S$. 
Then there is a Cartesian diagram
\[\xymatrix{
X\times_S Y \ar[r]^{i_{XY}}\ar[d] & X\times Y\ar[d]\\
\Delta_S \ar[r]^i & S\times S
,}
\]
where $\Delta_S\subseteq S\times S$ denotes diagonal. Moreover, we have isomorphisms
$\omega_{X\times_S Y}\otimes i^*_{XY}\omega_{X\times Y}^{-1}\cong\omega_{i_{XY}} \cong \omega_i\cong \omega_S^{-1}$.
Thus there is, for any $T\in\calA_0(X\times_SY/X)$, a pushforward homomorphism 
\begin{align*}
(i_{XY})_*\colon A^{\dim_S Y}_{T}(X\times_S Y,\omega_{Y/S})\to 
A^{\dim Y}_{i_{XY}(T)}(X\times Y, \omega_Y ).
\end{align*}
Passing to the colimit, we obtain a map $\res_S\colon\Cor^A_S(X,Y)\to\Cor^A_k(X,Y)$. 
To show that this homomorphism preserves composition in the category $\Cor_S$, first note that the commutative diagram
\[\begin{tikzcd}
X\times_S Y\times_S Z\ar{d}[swap]{p_{X\times_S Z}}\ar{r}{i_{XYZ}} & X\times Y\times Z\ar{d}{p_{XZ}}\\
X\times_S Z\ar{r}{i_{XZ}} & X\times Z
\end{tikzcd}\]
yields $(i_{XZ})_*(p_{X\times_S Z})_*=(p_{XY})_*(i_{XYZ})_*$. By decomposing the morphism $i_{XYZ}$ as
\[
i_{XYZ}\colon X\times_S Y\times_S Z\xrightarrow{i_X}X\times Y\times_S Z\xrightarrow{i_Y}X\times Y\times Z
\]
and applying the projection formula twice, we obtain the claim.
Hence the maps $\res_S$ above define a functor $\res_S\colon\Cor^A_S\to\Cor^A_k$. 

\subsubsection{}For any $X\in\Sm_S$, $Y\in\Sm_k$ and any admissible subset $T$ of $X\times Y$ we have a natural isomorphism $A^{\dim Y}_T(X\times Y,\omega_Y)\cong A^{\dim_SY_S}_T(X\times_SY_S,\omega_{X\times_SY_S/X})$. These isomorphisms define a natural isomorphism $\Cor^A_k(X,Y)\cong\Cor^A_S(X,Y_S)$. Similarly as in \cite[§6.2]{Calmes-Fasel} we deduce from this that the functors $\res_S$ and $\ext_S$ form an adjunction
$
\res_S:\Cor^A_S\rightleftarrows\Cor^A_k:\ext_S.
$

\subsection{Examples of cohomological correspondence categories}
Different choices for the cohomology theory $A^*$ recover various known correspondence categories, as well as new ones. For example, if  $A^*=\CH^*$ is the theory of Chow groups, then the definition of $\Cor^A_k$ gives back Voevodsky's category $\Cor_k$ of finite correspondences. If the ground field $k$ is perfect and of characteristic not $2$, then we can let $A^*$ be Chow--Witt theory, i.e., $A^*=\wt\CH^*$. In this case we obtain Calmès--Déglise--Fasel's category $\wt\Cor_k$ of finite Milnor--Witt correspondences. On the other hand, we can also define a good cohomology theory $A^*$ by letting $A^n_T(X,\scrL)\defeq \rmH^n_T(X,\bfI^{n},\scrL)$, where $\bfI^n$ is the Nisnevich sheaf of powers of the fundamental ideal. Then $\Cor^A_k$ is the category $\mathrm{W}\mkern-3mu\Cor_k$ of finite Witt-correspondences considered in \cite[Remark 5.16]{Calmes-Fasel}. Note that $\rmW\mkern-3mu\Cor_k$ thus defined differs from the category of Witt correspondences defined in \cite{Druzhinin}; however, arguing similarly as in \cite{effectivity} one can show that the associated derived categories of motives are equivalent after inverting the exponential characteristic of the ground field.

\subsubsection{Algebras over $\MSL$}\label{section:SL-orient}
More generally, we claim that any ring spectrum $E\in\SH(k)$ that is an algebra over $\MSL$ defines a cohomological correspondence category. Here $\MSL\in\SH(k)$ denotes the ring spectrum constructed by Panin and Walter in \cite{Panin-Walter}. 

In order to show this, let us first recollect a few notions from the formalism of six functors. Let $X\in\Sm_k$, and suppose that $i\colon Z\subseteq X$ is a closed subscheme. Let moreover $p\colon X\to\Spec k$ be the structure map. We then have adjunctions
$
p^*:\SH(k)\rightleftarrows\SH(X):p_*
$ and 
$i_!:\SH(Z)\rightleftarrows\SH(X):i^!.
$
If $q\colon\calE\to X$ is a vector bundle on $X$, let $s\colon X\to\calE$ denote the zero section. Recall from \cite[§5.2]{Hoyois-six-fu} that this defines \emph{Thom transformations}
\[
\Sigma^\calE\defeq q_\#s_*:\SH(X)\rightleftarrows \SH(X): s^!q^*\eqdef\Sigma^{-\calE}.
\]
In fact, these functors are defined for any $\xi\in \rmK(X)$ \cite[§16.2]{norms}.

\begin{definition}[\protect{\cite{MW-ring-spt,five-authors2}}]
Let $E\in\SH(k)$ be a spectrum and let $X$, $Z$ be as above. Let furthermore $\xi\in \rmK(Z)$. The \emph{$\xi$-twisted cohomology of $X$ with support on $Z$ and coefficients in $E$} is the space
\[
E_Z(X,\xi)\defeq\Map_{\SH(k)}(\sspt_k,p_*i_!\Sigma^\xi i^!p^*E),
\]
where $\sspt_k\in\SH(k)$ denotes the motivic sphere spectrum. The associated bigraded \emph{twisted cohomology groups with support} are then given as
\[
E_Z^{p,q}(X,\xi)\defeq[\sspt_k,\Sigma^{p,q}p_*i_!\Sigma^\xi i^!p^*E]_{\SH(k)}.
\]
\end{definition}

\begin{prop}
Suppose that $E\in\SH(k)$ is an $\MSL$-algebra. Let $X\in\Sm_k$, and suppose that $i\colon Z\subseteq X$ is a closed subscheme. For any line bundle $\scrL$ on $X$, set
\[
A^n_Z(X,\scrL)\defeq E^{2n,n}_Z(X,i^*\scrL)
\]
Then $A^*_Z(X,\scrL)$ defines a good cohomology theory and hence a cohomological correspondence category $\Cor^E_k$.
\end{prop}

\begin{proof}
The proposition follows from the six operations on $\SH(k)$, as explained in \cite{MW-ring-spt,fund-classes} or \cite{five-authors2}. Indeed, for the contravariant functoriality we refer to \cite[§2.2]{MW-ring-spt}, and for the definition of the cup product, see \cite[§2.3.1]{MW-ring-spt}. The pushforward is given by the Gysin map
$
f_!\colon E_Z(X,f^*\xi+\LL_f)\to E_{f(Z)}(Y,\xi),
$
where $\LL_f\in \rmK(X)$ is the cotangent complex of $f$; see \cite{fund-classes,five-authors2}. In particular, for $\MSL$ we have the Thom isomorphism $\Sigma^\xi\MSL\simeq \Sigma^{2\rk\xi,\rk\xi}\Sigma^{\det\xi-\calO}\MSL$ \cite[Example 16.29]{norms}. When $\xi$ is a line bundle $\scrL$, this gives the pushforward
$
f_*\colon A^n_Z(X,\omega_f\otimes f^*\scrL)\to A^{n-d}_{f(Z)}(Y,\scrL).
$
For the base change and projection formulas, see \cite[Proposition 2.2.5]{MW-ring-spt} and \cite[Remark 2.3.2]{MW-ring-spt}.
\end{proof}

\subsection{Presheaves on \texorpdfstring{$\Cor^A_k$}{CorAk}}Our basic object of study is the $\infty$-category $\PSh_\Sigma(\Cor^A_k;\Z)$ of presheaves of abelian groups on $\Cor^A_k$ that take finite coproducts to finite products. More generally we may of course also consider, for any coefficient ring $R$, the $\infty$-category $\PSh_\Sigma(\Cor^A_k;R)$ of presheaves of $R$-modules. For notational simplicity we will however mostly work with $R=\Z$. 

\begin{definition}
The objects of $\PSh_\Sigma(\Cor^A_k;\Z)$ will be referred to as \emph{presheaves with $A$-transfers}.

A presheaf with $A$-transfers $\scrF\in\PSh_\Sigma(\Cor_k^A;\Z)$ is \emph{homotopy invariant} if for any $X\in\Sm_k$, the map
$
\pr^*\colon\scrF(X)\xrightarrow{\cong}\scrF(X\times\A^1)
$
induced by the projection $\pr\colon X\times\A^1\to X$ is an isomorphism.
\end{definition}

\subsubsection{}The $\infty$-category $\PSh_\Sigma(\Cor^A_k;\Z)$ inherits a symmetric monoidal structure from that on $\Cor^A_k$ via Day convolution. Moreover, the graph functor $\gamma_A\colon\Sm_k\to\Cor^A_k$ defines a ``forgetful'' functor 
$
\gamma^A_*\colon \PSh_\Sigma(\Cor^A_k;\Z)\to\PSh_\Sigma(\Sm_k)
$
given by $\gamma^A_*(\scrF)\defeq \scrF\circ\gamma_A$. Similarly as in \cite[§1.2]{MW-cplx}, we deduce that the functor $\gamma^A_*$ admits a left adjoint $\gamma_A^*$ which is symmetric monoidal.

\subsubsection{Sheaves on $\Cor^A_k$}For any Grothendieck topology $\tau$, we define the $\infty$-category $\Shv_\tau(\Cor^A_k;\Z)$ consisting of those presheaves $\scrF\in\PSh_\Sigma(\Cor^A_k;\Z)$ such that $\gamma^A_*(\scrF)$ is a $\tau$-sheaf on $\Sm_k$. The adjunction $(\gamma_A^*,\gamma^A_*)$ above then defines an adjunction
\[
\gamma^*_A:\Shv_\tau(\Sm_k)\rightleftarrows\Shv_\tau(\Cor^A_k;\Z):\gamma_*^A,
\]
and the symmetric monoidal structure on $\PSh(\Cor^A_k;\Z)$ restricts to a symmetric monoidal structure on $\Shv_\tau(\Cor^A_k;\Z)$. 

\subsubsection{}In this text, we will almost exclusively work with the case when $\tau=\Nis$ is the Nisnevich topology. We show below (see \Cref{thm:ext}) that the full inclusion $i\colon\Shv_\Nis(\Cor^A_k;\Z)\to\PSh_\Sigma(\Cor^A_k;\Z)$ admits a left adjoint $a_\Nis\colon\PSh_\Sigma(\Cor^A_k;\Z)\to\Shv_\Nis(\Cor^A_k;\Z)$. In particular, the Nisnevich sheafification of a presheaf on $\Cor^A_k$ comes equipped with $A$-transfers in a canonical way. Hence we can make the following definition:

\begin{definition}
Let $X\in\Sm_k$ be a smooth $k$-scheme. Following the notation of \cite{Calmes-Fasel}, we let $\rmc_A(X)\in\PSh_\Sigma(\Cor^A_k;\Z)$ denote the representable presheaf on $\Cor^A_k$ given by $U\mapsto\Cor^A_k(U,X)$. Moreover, we let 
\[
\Z_A(X)\defeq a_\Nis(\rmc_A(X))\in\Shv_\Nis(\Cor^A_k;\Z)
\] 
denote the Nisnevich sheaf associated to the presheaf $\rmc_A(X)$.
\end{definition}

\subsection{Correspondences of pairs}
In the excision theorems of Sections \ref{section:rel-aff-line} and \ref{section:et-exc} we are always in the setting of a pair of schemes $j\colon U\subseteq X$, and we are led to consider the associated quotient $\coker(j^*\colon\scrF(X)\to \scrF(U))$ for a given presheaf with $A$-transfers. In particular, if $U=X$ and $j$ is the identity, then the associated quotient is zero. The notion of a correspondence of pairs provides a natural setting to study these objects.

\begin{definition}
Let $\Cor^{A,\pair}_S$ denote the category whose objects are those of $\SmOp_S$ and whose morphisms are defined as follows. For $(X,U),(Y,V)\in\SmOp_S$, with open immersions $j_X\colon U\to X$ and $j_Y\colon V\to Y$, consider the complex
\[
\Cor_S^A(X,V)\xrightarrow{d_0}\Cor_S^A(X,Y)\oplus\Cor_S^A(U,V)\xrightarrow{d_1}\Cor_S^A(U,Y)
\]
in which $d_0\defeq((j_Y)_*,j_X^*)$ and $d_1\defeq j_X^*-(j_Y)_*$.
We define the group $\Cor_S^{A,\pair}((X,U),(Y,V))$ of \emph{finite relative $A$-correspondences of pairs} as the homology of this complex, i.e.,
\[
\Cor_S^{A,\pair}((X,U),(Y,V))\defeq\ker d_1/\im d_0.
\]
In particular, if $U=X$, then $\Cor_S^{A,\pair}((X,X),(Y,V))=0$.
We denote the elements of $\Cor^{A,\pair}_S((X,U),(Y,V))$ by $(\alpha,\beta)$, where $\alpha\in\Cor_S^A(X,Y)$ and $\beta\in\Cor_S^A(U,V)$.
If $\beta$ is implicitly understood, we may write simply $\alpha$ instead of $(\alpha,\beta)$.
The composition in $\Cor^{A,\pair}_S$ is defined by $(\alpha,\beta)\circ(\gamma,\delta)\defeq(\alpha\circ\gamma,\beta\circ\delta)$.

Finally, we define the homotopy category $\ol\Cor^{A,\pair}_S$ of $\Cor^{A,\pair}_S$ as follows. The objects of $\ol\Cor^{A,\pair}_S$ are the same as those of $\Cor^{A,\pair}_S$, and the morphisms are given by
\begin{align*}
&\ol\Cor^{A,\pair}_S((X,U),(Y,V))\defeq\Cor^{A,\pair}_S((X,U),(Y,V))/\sim_{\A^1}\\
&=\coker\p*{\Cor^{A,\pair}_S(\A_S^1\times_S(X,U),(Y,V))\xrightarrow{i_0^*-i_1^*}\Cor^{A,\pair}_S((X,U),(Y,V))}.
\end{align*}
Here $\A_S^1\times_S(X,U)$ is shorthand for $(\A_S^1\times_S X,\A_S^1\times_S U)$. If $(\alpha,\beta)\in\Cor^{A,\pair}_S((X,U),(Y,V))$ is a finite relative $A$-correspondence of pairs, we write $[(\alpha,\beta)]$, or simply $[\alpha]$, for the image of $(\alpha,\beta)$ in $\ol\Cor^{A,\pair}_S((X,U),(Y,V))$. 
\end{definition}

\subsection{Correspondences between essentially smooth schemes}

We will frequently encounter local-, and henselian local schemes, and we need to consider correspondences also between such objects. The definitions and results below take care of this. We remind the reader that the definition of an étale neighborhood can be found in \Cref{def:et-nbhd} in the appendix.

\begin{definition}
Let $X=\varprojlim X_\alpha\in\EssSm_S$ be an essentially smooth $S$-scheme. Consider a closed subscheme $T=\varprojlim T_\alpha$ of $X$, where 
$T_\alpha$ is a closed subscheme of $X_\alpha$ for each $\alpha$. Define 
\[
A^{n}_T(U\times_S X,\omega_X)\defeq \varinjlim_\alpha A^{n}_{T_\alpha}(U\times_S X_\alpha,\omega_{X_\alpha}).
\]
Furthermore, for any $U=\varprojlim_\alpha U_\alpha\in \EssSm_S$, and for any $X\in \Sm_S$, we define 
\[
\Cor^A_S(U,X)\defeq\varinjlim_\alpha \Cor^A_S(U_\alpha,X).
\]
Finally, for any $X\in \Sm_S$, any point $x\in X$, and any $U\in \EssSm_S$, we put
\[
\Cor^A_S(U,X^h_x)\defeq\varprojlim_v\limits \Cor^A_S(U,X^\prime).
\]
Here the limit ranges over all étale neighborhoods $v\colon (X^\prime,x)\to (X,x)$ of $x$ in $X$.
\end{definition}

\begin{lemma}\label{lm:AEssSmSum}
For any $X\in \Sm_S$ of relative dimension $d$ over $S$, and for any henselian local scheme $U\in \EssSm_S$, 
we have 
\[
A^{d}_T(U\times_S X,\omega_X)= \bigoplus_{x\in X} A^{d}_{T_x}(U\times_S X_x,\omega_X)= \bigoplus_{x\in X} A^{d}_{T^h_x}(U\times_S X^h_x,\omega_{X^h_x})
\]
for any $T\in\calA_0(U\times_S X/U)$. Here $x$ ranges over the set of all (not necessarily closed) points of $X$, and $T_x\defeq T\times_X X_x$; $T^h_x\defeq T\times_X X^h_x$.
\end{lemma}

\begin{proof}
Since $U$ is henselian local and $T\in\calA_0(U\times_S X/U)$ is finite over $U$, it follows that $T$ is a semi-local henselian scheme. In fact,
$T= \coprod\limits_{z\in T_{(0)}} T_z^h$, where $z$ ranges over the set of closed points in $T$. 
Hence $T=\coprod\limits_{x\in X} T_x$ and $T= \coprod\limits_{x\in X} T^h_x$, 
where $x$ ranges over the set of all points of $X$. In particular we have $T_x=T^h_x$. 
We note that $T^h_x$ is semi-local henselian, but not necessarily local. By Zariski excision, we obtain
$
A^{d}_T(U\times_S X,\omega_X)= \bigoplus_{x\in X} A^{d}_{T_x}(U\times_S X,\omega_X),
$
and 
$
A^{d}_{T_x}(U\times_S X^\prime,\omega_X)=A^{d}_{T_x}(U\times_S X,\omega_X)
$ 
for any open $X^\prime\subseteq X$ containing $x$. This implies the first claim.

For the second equality, note that since the scheme $T^h_x$ is semi-local henselian for any $x\in X$, it follows that $T^h_x$ is isomorphic to its preimage under any étale neighborhood $v\colon (X^\prime,x)\to (X,x)$. Hence it follows from étale excision that 
$A^{d}_{T^h_x}(U\times_S X,\omega_X)=A^{d}_{T^h_x}(U\times_S X^\prime,\omega_X)$, and consequently 
$
A^{d}_{T^h_x}(U\times_S X,\omega_X)=A^{d}_{T^h_x}(U\times_S X^h_X,\omega_{X^h_X}).
$
So the second equality follows.
\end{proof}

\begin{lemma}\label{lm:EssSmACor=limA_T}
Let $X\in\Sm_S$ be as in \Cref{lm:AEssSmSum}. Then, for any point $x\in X$ and for any henselian local scheme $U\in \EssSm_S$ we have 
\begin{align*}
\Cor^A_S(U,X_x) &= \varinjlim\limits_{T\in \calA_0(U\times_S X_x/U), x\in X} A^{d}_T(U\times_S X_x,\omega_{X_x}),\\
\Cor^A_S(U,X^h_x) &= \varinjlim\limits_{T\in \calA_0(U\times_S X^h_x/U), x\in X} A^{d}_T(U\times_S X^h_x,\omega_{X^h_x}).
\end{align*}

\end{lemma}

\begin{proof}
The first claim follows from the first equality of \Cref{lm:AEssSmSum}, by the following computation:
\begin{align*}
\Cor^A_S(U,X_x) &= 
\varprojlim_{v} 
\varinjlim\limits_{T\in \calA_0(U\times_S X^\prime/U)} A^{d}_T(U\times_S X^\prime,\omega_{X^\prime})\\
&= \varprojlim_{v} 
\varinjlim\limits_{T\in \calA_0(U\times_S X^\prime/U)} 
\bigoplus_{x^\prime\in X^\prime} A^{d}_{T_{x^\prime}}(U\times_S X^\prime_{x^\prime},\omega_{X^\prime_{x^\prime}})\\
&=\varprojlim_{v} 
\bigoplus_{x^\prime\in X^\prime} \varinjlim\limits_{T\in \calA_0(U\times_S X^\prime_{x^\prime}/U)} 
A^{d}_{T}(U\times_S X^\prime_{x^\prime},\omega_{X^\prime_{x^\prime}})\\
&=\varprojlim_{v} 
\varinjlim\limits_{T\in \calA_0(U\times_S X^\prime_x/U)} 
A^{d}_{T}(U\times_S X^\prime_{x},\omega_{X^\prime_{x}})= A^{d}_{T}(U\times_S X_{x},\omega_{X_{x}}).
\end{align*}
Here $v\colon (X^\prime,x)\hookrightarrow (X,x)$ ranges over the set of Zariski neighborhoods of $x$ in $X$.
The second equality of the claim follows in a similar manner from the second equality of \Cref{lm:AEssSmSum} with $X_x$ replaced by $X^h_x$, and with $v$ ranging over the set of étale neighborhoods of $x$ in $X$.
\end{proof}

\subsection{Constructing correspondences from functions and trivializations}
From now on we will assume that the base scheme $S$ is the spectrum of a field $k$. Later on we will also have to put more restrictions on $k$ (e.g., infinite or perfect); the appropriate assumptions will be stated in the beginning of each section where they are needed.

\subsubsection{}We will now describe how to construct a finite $A$-correspondence from the data of a regular function on a relative curve together with a trivialization of the relative canonical class. This construction can be thought of as an analogous statement to the defining axiom of a pretheory in the sense of Voevodsky \cite{Voe-hty-inv}, and will be used throughout.

\begin{construction}\label{constr:lrangleclass}
Suppose that there is a diagram
\begin{equation}
\label{diag:pCf}
\begin{tikzcd}
\calC\ar{r}{f}\ar{d}[swap]{p}\ar{dr}{g} & \A^1\\
U & X
\end{tikzcd}
\end{equation}
in $\Sm_k$ satisfying the following properties:
\begin{enumerate}
\item $p\colon \mathcal C\to U$ is a smooth relative curve, and $g\colon\calC\to X$ is any morphism.
\item $Z(f)=Z\amalg Z^\prime$, with $Z$ finite over $U$.
\item There is an isomorphism $\mu\colon \calO_\calC\xrightarrow{\cong}\omega_{\mathcal C/U}$.
\end{enumerate} 
We can then define finite $A$-correspondences 
\begin{align*}
&\Div^A_U(f)_Z^\mu\in \Cor^A_U(U,\mathcal C);\quad \Div^A(f)_Z^\mu\in \Cor^A_k(U,\mathcal C);\\
&\Div^A_U(f)_Z^{\mu,g}\in \Cor^A_U(U,X);\quad \Div^A(f)_Z^{\mu,g}\in \Cor^A_k(U,X)
\end{align*}
as follows:

Let $\Gamma_f$ denote the graph of the morphism $f$, 
with embedding $i_f\colon \Gamma_f\hookrightarrow  \mathcal C\times\A^1$.
Consider the pushforward homomorphism 
$(i_f)_*\colon A^0(\Gamma_f,\mathcal O_{\Gamma_f}\otimes\omega_{i_f})\to A_{\Gamma_f}^{1}(\mathcal C\times\A^1,\mathcal O_{\mathcal C\times\A^1})$, and 
let $dT\colon \mathcal O_{\A^1}\cong \omega_{\A^1}$ be the trivialization defined by the coordinate function $T$ on $\A^1$. Using the trivializations $-dT$ and $\mu$ we then obtain a homomorphism 
$
i_*\colon A^0(\Gamma_f,\mathcal O_{\Gamma_f})\to A_{\Gamma_f}^{1}(\mathcal C\times \A^1,\omega_{\mathcal C\times\A^1/U\times\A^1}).
$
Consider the image $i_*(1)\in A_{\Gamma_f}^{1}(\mathcal C\times\A^1,\omega_{\calC\times\A^1/U\times\A^1})$ of $1\in  A^0(\Gamma_f,\calO_{\Gamma_f})$ under the map $i_*$.

Next we may pull back along the zero section,
$i_0^*\colon A_{\Gamma_f}^{1}(\mathcal C\times\A^1,\omega_{\mathcal C\times\A^1/U\times\A^1})\to A_{Z(f)}^{1}(\mathcal C,\omega_{ \mathcal C/U})$. Since $Z(f)=Z\amalg Z^\prime$ we have 
$
A_{Z(f)}^{1}(\mathcal C,\omega_{ \mathcal C/U})= A_{Z}^{1}(\mathcal C,\omega_{ \mathcal C/U})\oplus A_{Z^\prime}^{1}(\mathcal C,\omega_{\mathcal C/U})
$
by \Cref{rem:Zarex}. We define the finite relative $A$-correspondence
\[
\Div^A_U(f)_Z^\mu\in\Cor^A_U(U,\calC)
\]
as the image of $i_*(1)\in A^1_{\Gamma_f}(\calC\times\A^1,\omega_{\calC\times\A^1/U})$ under the composite homomorphism
\[
A_{\Gamma_f}^{1}(\mathcal C\times\A^1,\omega_{\mathcal C/U})\xrightarrow{i_0^*}
A_{Z(f)}^{1}(\mathcal C,\omega_{ \mathcal C/U})\to
A_{Z}^{1}(\mathcal C,\omega_{ \mathcal C/U})\to\Cor^A_U(U,\calC).
\]
Here the second map is the projection to the first coordinate, and the last map is the canonical homomorphism to the colimit. By composing with the morphism $g$ we obtain the finite relative $A$-correspondence
\[
\Div_U^A(f)_Z^{\mu,g}\defeq g\circ\Div_U^A(f)_Z^\mu\in\Cor_U^A(U,X).
\]
We readily obtain a nonrelative $A$-correspondence by applying the functor $\res_U$. More precisely, we define
\[
\Div^A(f)_Z^{\mu,g}\defeq g\circ \res_U(\Div^A_U(f)_Z^\mu)\in \Cor^A_k(U,X).
\]
If it is clear from the context, we might drop the trivialization $\mu$ or the map $g$ from the notation. Moreover, if $Z=Z(f)$ and $Z$ is finite over $U$, we may also abbreviate $\Div^A(f)_{Z(f)}$ to $\Div^A(f)$. We think of $\Div^A(f)_Z^\mu$ as a divisor supported on $Z$ whose multiplicity at each component of $Z$ is given by an $A$-cohomology class.
\end{construction}

\begin{lemma}\label{lm:fmu}
Let $\mathcal C$, $Z$, $p$, $f$ and $g$ be as in \Cref{constr:lrangleclass}. Then
$\Div^A(\lambda f)^{\lambda \mu, g}_Z=\Div^A(f)^{\mu, g}_Z$
for any $\lambda\in \Gamma(U,\calO_U^\times)$.
\end{lemma}

\begin{proof}
For any smooth $U$-scheme $X$, any closed subscheme $Z\subseteq X$, and any line bundle $\scrL$ on $X$,
define the automorphism $\Lambda_X\colon A^*_Z(X,\scrL)\to A^*_Z(X,\scrL)$ as the map induced by the automorphism $\scrL\to \scrL$ given by multiplication by $\lambda$.

Consider the homomorphisms
\begin{align*}
i_*\colon 
A^0(\Gamma_f,\mathcal O_{\Gamma_f})&\to A_{\Gamma_f}^{1}(\mathcal C\times \A^1,\omega_{\mathcal C\times\A^1/U}),
\\
i_*^{\lambda f,\lambda \mu}\colon 
A^0(\Gamma_{\lambda f},\mathcal O_{\Gamma_{\lambda f} })&\to A_{\Gamma_{\lambda f} }^{1}(\mathcal C\times \A^1,\omega_{\mathcal C\times\A^1/U})
\end{align*}
in the constructions of 
$\Div^A(f)^{\mu, g}_Z$ and $\Div^A(\lambda f)^{\lambda \mu, g}_Z$.
Let moreover $i_*^{\lambda\mu}$ denote the homomorphism 
\[
i_*^{\lambda\mu}\colon
A^0({\Gamma_f},\mathcal O_{\Gamma_f})\to A_{\Gamma_f}^{1}(\mathcal C\times \A^1,\omega_{\mathcal C\times\A^1/U})
\]
given by the trivialization $dT\otimes\lambda\mu$.
Define automorphisms 
\begin{align*}
H^{\lambda}\colon \A^1\times \calC\to \A^1\times \calC,&\quad (T,x)\mapsto (\lambda T,x),\\
H^{\lambda^{-1}}\colon \A^1\times \calC\to \A^1\times \calC,&\quad (T,x)\mapsto (\lambda^{-1}T,x).
\end{align*}
Then $H^{\lambda^{-1}}(\Gamma_{\lambda f})=\Gamma_f$,
and $H_*^{\lambda^{-1}}(dT)=\lambda^{-1} dT$.
Hence 
\[
H_*^{\lambda^{-1}} \circ i_*^{\lambda f,\lambda\mu} = (\Lambda_{\mathcal C\times\A^1})^{-1} \circ i_*^{\lambda\mu} = i_*,
\]
and the claim follows.
\end{proof}

\begin{lemma}
Let $\mathcal C$, $Z$, $p$ and $f$ be as in \eqref{diag:pCf}
and suppose that $Z = Z_1\amalg Z_2$ with both $Z_1$ and $Z_2$ finite over $U$.
Then $\Div^A(f)_Z^{\mu,g}=\Div^A(f)_{Z_1}^{\mu,g}+\Div^A(f)_{Z_2}^{\mu,g}$.
\end{lemma}

\begin{proof}
The claim follows from the definition and \Cref{rem:Zarex}.
\end{proof}

\begin{definition}\label{ex:EtneigShNot}
Let $\mathcal C$, $U$, $\mu$, $Z$, $X$, $p$, $f$ and $g$ be as above and suppose that $U^\prime\subseteq U$ and $X^\prime\subseteq X$ are open subschemes
such that $Z\times_U U^\prime\subseteq g^{-1}(X^\prime)$. Write $f^\prime\defeq f\big|_{\mathcal C\times_U U^\prime}$ and $g^\prime\defeq g\big|_{\mathcal C\times_U U^\prime}$. This data defines a correspondence of pairs
\[
\p*{\Div^A(f)^{\mu,g}_{Z},\Div^A(f^\prime)^{\mu,g^\prime}_{Z\times_U U^\prime}}\in\Cor^{A,\pair}_k( (U,U^\prime),(X,X^\prime)).
\]

Suppose furthermore that
$
\pi\colon (\calC^\prime,Z^\prime)\to (\calC,Z)
$
is an étale neighborhood (see \Cref{def:et-nbhd}) 
satisfying $Z^\prime \times_U U^\prime \subseteq v^{-1}(X^\prime)$, where $v\defeq g \circ \pi$.
Then this data defines 
a finite $A$-correspondence of pairs
$\Div^A(\tilde{f})_Z^{\tilde\mu, v}\in\Cor^{A,\pair}_k((U,U^\prime),(X,X^\prime))$, where $\tilde f\defeq\pi^*(f)$ and $\tilde \mu\defeq\pi^*(\mu)$.
If the morphism $\pi$ is implicitly understood from the context, we may sometimes abuse notation and write simply
$\Div^A(f)_Z^{\mu, g}$ for this $A$-correspondence. 
\end{definition}

\begin{lemma}\label{lm:ZeroPairC}
Let $\mathcal C$, $U$, $\mu$, $Z$, $X$, $p$, $f$ and $g$ be as above and suppose that $U^\prime\subseteq U$ and $X^\prime\subseteq X$ are open subschemes.
If $Z\cap g^{-1}(X\setminus X^\prime)=\varnothing$, then $\Div^A(f)^{\mu,g}_{Z}=0\in \Cor^{A,\pair}_k((U,U^\prime),(X,X^\prime))$.  
\end{lemma}
\begin{proof}
The correspondence $\Div^A(f)^{\mu,g}_{Z}\in \Cor^{A}_k(U,X^\prime)$
defines the diagonal 
in the diagram
\[\begin{tikzcd}
X^\prime\ar[r] & X\\
U^\prime\ar[r]\ar[u]& U.\ar[u]\ar[ul]
\end{tikzcd}\]
Moreover, the vertical arrows in the above diagram define the correspondence of pairs $\Div^A(f)_Z^{\mu, g}\in\Cor^{A,\pair}_k((U,U^\prime),(X,X^\prime))$; it follows that $\Div^A(f)^{\mu,g}_Z$ factors through $(X',X')$ and is therefore zero.
\end{proof}

\begin{definition}\label{def:ip}
Let $U\in\Sm_k$ and suppose that $\lambda$ is an invertible regular function on $U$. We can then consider the morphism $(\id,\lambda)\colon(U\times U,\omega_U)\to(U\times U,\omega_U)$ in $\SmOp_k^\rmL$. We denote by 
\[\langle \lambda\rangle\in \Cor^A_k(U,U)\] 
the image of $\id_U\in\Cor^A_k(U,U)$ under the corresponding pushforward map $(\id,\lambda)_*$. In particular, if $\lambda=-1$, we will write $\epsilon$ for the finite $A$-correspondence 
$
\epsilon\defeq-\ip{-1}\in\Cor^A_k(U,U).
$
\end{definition}

\begin{example}
Suppose that $A^*=\wt\CH^*$, so that $\Cor^A_k$ is the category of finite Milnor--Witt correspondences. Then $\ip\lambda\in\Cor^A_k(U,U)$ is the Milnor--Witt correspondence $\ip\lambda\cdot\id_U\in\wt\Cor_k(U,U)$ given by multiplication with the quadratic form $\ip\lambda\in\bfK_0^{\MW}(U)$. In particular, the finite $A$-correspondence $\epsilon$ coincides with the usual $\epsilon$ defined in Milnor--Witt $\rmK$-theory.
\end{example}

\begin{lemma}\label{lm:sectZ(f)cor}
Let $U$, $\mathcal C$, $p$, $f$ and $g$ be as in \eqref{diag:pCf}. Suppose also that $p$ induces an isomorphism $Z(f)\cong U$, so that $Z(f)$ defines a section $s\colon U\to \calC$ of $p$.
Then the following hold: \begin{itemize}
\item[(a)] 
There is an invertible regular function $\lambda$ on $U$ such that $\Div^A(f)_{Z(f)}^{\mu,g}=g\circ s\circ \ip\lambda$ in $\Cor^A_k(U,X)$.
\item[(b)] 
If moreover $\mu|_{Z(f)}= df$, where $df$ denotes the trivialization of the normal bundle $N_{Z(f)/\calC}$ defined by $f$, then $\Div^A( f)_{Z(f)}^{\mu,g}=g\circ s$.
\end{itemize}
\end{lemma}

\begin{proof}
(a) 
Let $j\colon Z(f)\to \Gamma_f$, $j_f\colon Z(f)\to\calC$ and $i_f\colon\Gamma_f\to\calC\times\A^1$ denote the closed embeddings. Consider the following diagram consisting of two 
squares of varieties equipped with compatible sets of line bundles (in which we have also included the relevant line bundles in the notation): 
\[\begin{tikzcd}
(Z(f),\calO_{Z(f)})\ar{r}{(\id,\mu)} \ar{d}{j} & (Z(f),\omega_{Z(f)/U})\ar{r}{j_f}\ar{d}{j} & (\calC,\omega_{\calC/U})\ar{d}{i_0} \\
(\Gamma_f,\calO_{\Gamma_f})\ar{r}{(\id,\mu)} & (\Gamma_f,\omega_{\Gamma_f/U})\ar{r}{i_f} & (\calC\times\A^1,\omega_{\calC\times\A^1/U}).
\end{tikzcd}\]
The first square is evidently transversal (and strongly transversal). To prove that the second one is (strongly) transversal, it is enough to note that the homomorphism
$k[\calC][T]=k[\calC\times \A^1]\to k[\calC]$ given by $T\mapsto 0$ takes the function $f-T$ to $f$ and induces an isomorphism 
\[
N_{\Gamma_f/\calC\times \A^1}\otimes k[\calC\times 0]=(f-T)/(f-T)^2\otimes k[\calC][T]/(T)\cong (f)/(f)^2=N_{Z(f)/\calC}.
\]
Hence the base change axiom gives us the following commutative diagram:
\[
\begin{tikzcd}
A^0(\Gamma_f,\calO_{\Gamma_f})\ar{r}{\mu}\ar{d}[swap]{j^*} 
& A^0(\Gamma_f,\omega_{\Gamma_f/U})\ar{d}[swap]{j^*}\ar{r}{(i_f)_*} 
& A^1_{\Gamma_f}(\calC\times\A^1,\omega_{\calC\times\A^1/U})\ar{d}[swap]{i_0^*}  \\
A^0(Z(f),\calO_{Z(f)})\ar{r}{\mu_{Z(f)}}\ar{drr}[swap]{(j_f,\nu)_*} & 
A^0(Z(f),\omega_{Z(f)/U}\otimes\omega_{j})\ar{r}{(j_f)_*} &
A^1_{Z(f)}(\calC,\omega_{\calC/U}\otimes\omega_{i_0})\ar{d}{-dT}\\ 
& & A^1_{Z(f)}(\calC,\omega_{\calC/U}).
\end{tikzcd}\]
Here 
$j^*$ and $i_0^*$ are defined via the canonical isomorphisms 
$j^*(\omega_{\Gamma_f/U})\cong \omega_{Z(f)/U}\otimes\omega_j$
and
$i_0^*(\omega_{\calC\times\A^1/U})\cong \omega_{\calC/U}\otimes\omega_{i_0}$
induced by the short exact sequences of vector bundles
\[
0\to T_{Z(f)}\to j^*(T_{\Gamma_f})\to N_{Z(f)/\Gamma_f}\to0
\] and 
\[
0\to T_{\calC\times 0}\to i_0^*(T_{\calC\times\A^1})\to N_{\calC\times0/\calC\times\A^1}\to0.
\]
Moreover, the homomorphism $\mu_{Z(f)}$ is given as the composition of $\mu|_{Z(f)}$ and the isomorphism $j^*\omega_{\Gamma_f/U}\cong \omega_{Z(f)/U}\otimes\omega_j$;
the homomorphism $(j_f)_*$ is defined via the isomorphism $j_f^*(\omega_{i_0})\cong \omega_{j}$ induced by the canonical isomorphism $\Gamma_f\cong \calC$;
and the diagonal homomorphism $(j_f,\nu)_*$ is induced by some trivialization $\nu\colon \calO_{Z(f)}\cong\omega_{Z(f)/U}$.

It follows from the construction that
$\Div^A(f)_{Z(f)}^\mu=-dT(i_0^*(i_f)_*\mu(1))$.
Since the diagram is commutative we thus obtain 
$\Div^A(f)_{Z(f)}^\mu=(j_f,\nu)_*j^*(1)=s\circ\langle \lambda\rangle$, where $\lambda$ is given as the fraction of $\nu$ and the canonical isomorphism $\omega_{Z(f)/U}\cong \calO_U$ induced by the isomorphism $p\colon Z(f)\stackrel{\cong}{\to}U$.

(b)
A straightforward computation with isomorphisms of line bundles shows that 
$(j_f)_*$ is given as the product of the canonical isomorphism $\calO_{Z(f)}\cong \omega_{Z(f)/U}$ with the invertible function $\mu\big|_{Z(f)}\otimes df^{-1}$, where $df\colon \calO_{Z(f)}\cong \omega_{Z(f)/U}$ denotes the trivialization induced by the choice of the generator $-f$ of the ideal $(f)=I(Z(f))$.
So $\lambda=1$, and the claim follows.
\end{proof}

\subsection{Some homotopies}
We now give a computation with $A$-correspondences that will come in handy later on, especially in the proof of \Cref{lm:locup}.

\begin{lemma}\label{lm:sqroot}
Suppose that the base field $k$ is infinite. Let $U$ be an essentially smooth local scheme over $k$ and let 
$\lambda\in \Gamma(U,\calO_U^\times)$. 
Suppose that $\lambda = w^2$ for some invertible section $w$ on $U$.
Then $\langle \lambda \rangle \sim_{\A^1} \id_U\in \Cor^A_k(U,U)$.
Similarly $\langle \lambda \rangle \sim_{\A^1} \id_{(U,V)}\in \Cor^{A,\pair}_k((U,V),(U,V))$ for any open subscheme $V\subseteq U$.
\end{lemma}

\begin{proof}
Assume first that $V=\varnothing$ and $\lambda(x)\neq 1$, where $x\in U$ is the closed point.
Let $\alpha\defeq(\lambda-1)^{-1}$, and define the regular function 
\[
h\defeq(1-\nu) \alpha(t-\lambda)(t-1) + \nu  \alpha(t-w)^2\in 
\Gamma(\stackrel{t}{\G_m}\times U\times\stackrel{\nu}{\A^1},\calO).
\]
Keeping the notation of \eqref{diag:pCf} in mind, consider the following diagram:
\[\begin{tikzcd}
\G_m\times U\times\A^1\ar{d}[swap]{p}\ar{dr}{\pr}\ar{r}{h} & \A^1\\
U\times\A^1 & U.
\end{tikzcd}\]
Here the morphisms $p$ and $\pr$ are the projections.
We aim to apply \Cref{constr:lrangleclass} to this diagram. To this end, notice that $h$ is a polynomial in $t$ with leading term $\alpha$, which is invertible on $U$. Moreover, the substitution $t\mapsto 0$ takes $h$ to $(1-\nu) \alpha\lambda + \nu \alpha w^2=\alpha\lambda$, which is invertible too.
Hence $Z(h)\subseteq\G_m\times U\times \A^1$ is finite over $U\times\A^1$.
Using the trivialization $tdt$ of the canonical class of $\G_m$, we get from \Cref{constr:lrangleclass} a finite relative $A$-correspondence
\[
\Theta\defeq \Div^A_U(h)^{tdt,\pr}\in \Cor^A_U (U\times\A^1,U).
\]
Let $i_0,i_1\colon U\to U\times\A^1$ denote the zero-, and unit sections. We then have
\begin{align*}
\Theta\circ i_0 &= \Div^A_U( \alpha(t-\lambda)(t-1) )^{tdt, \pr}\\ 
&= \Div^A_U( (\lambda-1)^{-1}(t-\lambda)(t-1) )^{tdt,\pr}_{Z(t-\lambda)} + 
\Div^A_U( -(1-\lambda)^{-1}(t-\lambda)(t-1) )^{tdt,\pr}_{Z(t-1)} \\ 
&=\langle \lambda \rangle + \langle -1 \rangle.
\end{align*}
On the other hand,
\[
\Theta\circ i_1 = \Div^A_U( \alpha(t-w)^2 )^{tdt, \pr}
=\Div^A_U( \alpha w^{-1} (t-w)^2 )^{dt, \pr},
\]
where the second equality follows from \Cref{lm:fmu}.
Thus we see that
\[
\langle \lambda \rangle + \langle -1 \rangle \sim_{\A^1} \Div^A_U( \alpha w ^{-1}(t-w)^2 )^{dt, \pr}
\in \Cor^A_U(U,U).
\]
We now construct yet another homotopy similar to the one in the proof of \cite[Lemma 13.15]{hty-inv},
which is in turn inspired by \cite[Lemma 7.3]{Nesh-FrKMW}.
Put $\alpha^\prime\defeq\alpha w^{-1}$. Consider the regular function 
\[
h^\prime= (1-\nu) \alpha^\prime (t-w)^2 + \nu \alpha^\prime (t-{\alpha^\prime}^{-1})t\in 
\Gamma(\stackrel{t}{\A^1}\times U\times\stackrel{\nu}{\A^1},\calO),
\]
along with the diagram
\[\begin{tikzcd}
\A^1\times U\times\A^1\ar{d}[swap]{p'}\ar{dr}{\pr'}\ar{r}{h'} & \A^1\\
U\times\A^1 & U
\end{tikzcd}\]
in which $p'$ and $\pr'$ are the projections.
As $h'$ is a polynomial in $t$ with leading term $\alpha^\prime$, which is invertible on $U$, it follows that $Z(h^\prime)\subseteq\A^1\times U\times \A^1$ is finite over $U\times\A^1$. Using the trivialization $dt$ of the canonical class of $\A^1$, we then get from \Cref{constr:lrangleclass} a finite $A$-correspondence
\[
\Theta^\prime\defeq \Div^A_U(h^\prime)^{dt,\pr'}\in \Cor^A_U(U\times\A^1,U).
\]
By definition of $h'$, the $A$-correspondence $\Theta'$ satisfies
\begin{align*}
\Theta^\prime\circ i_0 &= \Div^A_U( \alpha^\prime (t-w)^2 )^{dt, \pr^\prime},\\
\Theta'\circ i_1 &= \Div^A_U( \alpha^\prime (t-{\alpha^\prime}^{-1})t )^{dt, \pr^\prime}\\
&= \Div^A_U( \alpha^\prime (t-{\alpha^\prime}^{-1})t)^{dt,\pr^\prime}_{Z(t-{\alpha^\prime}^{-1})}+
\Div^A_U( \alpha^\prime (t-{\alpha^\prime}^{-1})t )^{dt, \pr^\prime}_{Z(t)}\\
&=\langle 1 \rangle + \langle -1 \rangle. 
\end{align*}
Thus we see that
\[\Div^A_U( \alpha^\prime (t-w)^2 )^{dt, \pr^\prime} \sim_{\A^1} \langle 1 \rangle + \langle -1 \rangle
\in \Cor^A_U(U,U).\]
Now, since 
$\Div^A_U( \alpha^\prime (t-w)^2 )^{dt, \pr^\prime}=\Div^A_U( \alpha^\prime (t-w)^2 )^{dt, \pr}$,
we get
\[
\langle \lambda \rangle + \langle -1 \rangle\sim_{\A^1} \langle 1 \rangle + \langle -1 \rangle\in \Cor^A_U(U,U).
\]
Thus the claim follows from the fact that $\langle 1 \rangle=\id_U$.

We have now proved the claim in the case $\lambda(x)\neq 1$.
In the general case when $\lambda\in \Gamma(U,\calO_U^\times)$,
consider a function $u\in \Gamma(U,\calO_U^\times)$ such that $u(x)\neq w(x)^{-1}$ and $u(x)\neq 1$.
Such a function exists, since the base field is infinite by assumption.
Then we have by the above that 
$\langle \lambda u^2 \rangle \sim_{\A^1} \id_U$ and $\langle u^2 \rangle \sim_{\A^1} \id_U$ in $\Cor^A_U(U,U)$.
Thus, since $\langle \lambda u^2 \rangle=\langle \lambda \rangle\circ \langle u^2 \rangle$, the claim follows.

So the claim of the lemma is done for $V=\varnothing$.
The case of a pair $(U,V)$ with $V\neq \varnothing$ follows, since
all the constructed homotopies are relative homotopies over $U$, i.e., they are elements of $\Cor^A_U(U\times\A^1,U)$. Consequently all the homotopies defined are elements in $\Cor^{A,\pair}_U((U,V)\times\A^1,(U,V))$ as well.
\end{proof}

\section{Connection to framed correspondences}\label{section:framed}

Using similar techniques as in \Cref{constr:lrangleclass} we can define a functor $\Upsilon\colon\Fr_*(k)\to\Cor^A_k$ from the category of framed correspondences \cite{Framed} to the category $\Cor^A_k$. See also \cite{five-authors2} for an alternative approach using Thom classes \cite[Lemma 4.3.24]{five-authors2}. 

\begin{construction}\label{constr:functorFrAcor}
Let $\Phi=(Z,\calV, \phi; g)\in \Fr_n(X,Y)$
be an explicit framed correspondence. Thus 
$Z$ is a closed subset in $\A^n\times X=\A^n_X$; $(\calV,Z)\to (\A^n_X,Z)$ is an étale neighborhood of $Z$ in $\A^n_X$;
$\phi=(\phi_i)$, where the $\phi_i$'s are regular functions on $\calV$ such that $Z=Z(\phi)$; and $g$ is a morphism $g\colon \calV\to Y$.
For any unit $\lambda\in k^\times$ 
 we define a finite $A$-correspondence $\Upsilon_\lambda(\Phi)\in \Cor^A_k(X,Y)$ in the following way.

Let $dt\colon \omega_{\A^1}\cong \calO_{\A^1}$ denote the standard trivialization of the canonical class, and consider further two trivializations
 $\mu_1,\mu_2\colon \omega_{\A^n}\cong \calO_{\A^n}$ given by $\mu_1=(dt)^{\wedge n}$ and $\mu_2=\lambda^n\mu_1$.
Let $\Gamma$ denote the graph $\Gamma\subseteq \A^n_X\times_X\calV=\A^n\times\calV$ of the relative morphism $\calV\to \A^n_X$ over $X$.
Then there is a canonical projection $\Gamma\to \A^n_X$.
Denote by $i_X\colon X\to \A^n_X$ and $i_{\calV}\colon \calV \to \A^n\times \calV$ the embeddings given by the zero sections.
Let furthermore $g^\prime\colon \calV\to X\times Y$ denote the product of $g$ and the projection to $X$. The following diagram summarizes the situation:
\begin{equation}\label{diag:functorFrACor}\begin{tikzcd}
& Y&\\
\A^n\times\calV\ar[d] & \calV\ar[l,shift right,swap,"\Gamma"]\ar[l,shift left,"i_\calV"]\ar[d]\ar[r]\ar[u,"g"] &\A^n_X \\
\A^n\times X & X.\ar[l,swap,"i_X"]
\end{tikzcd}
\end{equation}
We then define
$
\Upsilon_\lambda(\Phi)\defeq g^\prime_*(i^*_{\calV}(\Gamma_*(1))),
$
where we use the trivialization $\mu_1$ of the canonical class $\omega_{\A^n}$, and the trivialization of $\omega_{\calV/X}$ defined by the pullback of $\mu_2$ along the étale morphism $\calV\to \A^n\times X$.

In other words, the finite $A$-correspondence $\Upsilon_\lambda(\Phi)$ is obtained as the image of $i_\calV^*(\Gamma_*(1))\in A^n_Z(\calV,\omega_{\calV/X})$ under the composition
\[
A^n_Z(\calV,\omega_{\calV/X})\to\Cor^A_X(X,\calV)\xrightarrow{\res_X}\Cor^A_k(X,\calV)\xrightarrow{g_*}\Cor^A_k(X,Y)
\]
in which the last map is given by composition with $g$.
\end{construction}

\begin{theorem}
For each unit $\lambda\in k^\times$, \Cref{constr:functorFrAcor} defines a functor 
$
\Upsilon_\lambda\colon\Fr_*(k)\to \Cor^A_k
$ 
that carries the framed correspondence $\sigma=(0,\A^1,t,\pr\colon \A^1\to \pt)\in \Fr_1(\pt,\pt)$ to $\ip\lambda\in\Cor^A_k(\pt,\pt)$. Moreover, $\Upsilon_\lambda$ factors through the category $\Z\rmF_*(k)$ of linear framed correspondences.
\end{theorem}

\begin{proof}
\Cref{constr:functorFrAcor} gives rise to a map
$\Upsilon_\lambda$ depending on the fraction $\lambda\in k^\times$ of the two trivializations of the canonical classes. To show that $\Upsilon_\lambda$ is in fact a functor, we need to check the following: 
\begin{itemize}
\item[(1)] Equivalent explicit framed correspondences give rise to the same finite $A$-correspondence.
\item[(2)] Let $\id_{X}\in \Fr_{0}(X,X)$ be the identity morphism in the graded category $\Fr_*(k)$. Then $\Upsilon_\lambda(\id_{X})$ is equal to the identity morphism in the category $\Cor^A_k$.
\item[(3)] For any $\Phi_1\in \Fr_{n_1}(X_1,X_2)$ and $\Phi_2\in \Fr_{n_2}(X_2,X_3)$, we have $\Upsilon_\lambda(\Phi_2\circ \Phi_1)=\Upsilon_\lambda(\Phi_2)\circ\Upsilon_\lambda(\Phi_1)$.
\item[(4)] For any $\Phi=(Z,\calV,\phi;g)\in \Fr_n(X_1,X_2)$ such that $Z=Z_1\amalg Z_2$, we have $\Upsilon_\lambda(\Phi)=\Upsilon_\lambda(Z_1,\calV,\phi;g)+\Upsilon_\lambda(Z_2,\calV,\phi;g)$.
\end{itemize}
All points are straightforward from the properties of the cohomology theory $A^*$.
\end{proof}

\begin{remark}Note that \Cref{thm:strict-hty-inv} on strict homotopy invariance of presheaves on $\Cor^A_k$ follows from the existence of a functor from framed correspondences to $\Cor^A_k$ along with the fact that this theorem holds for framed correspondences by work of Garkusha--Panin \cite{hty-inv}. Below we will however give an explicit proof not relying on framed correspondences.
\end{remark}

\section{Injectivity on the relative affine line}\label{section:inj-aff-line}
In this section we prove the following theorem, which is the first in a series of ingredients necessary to establish strict homotopy invariance (\Cref{thm:strict-hty-inv}):

\begin{theorem}\label{th:injRelA1}
  Let $U$ be an affine smooth $k$-scheme, and suppose that $V_1\subseteq V_2\subseteq\A^1_U$ are two open subschemes such that $\A^1_U\setminus V_2$ and $V_2\setminus V_1$ are finite over $U$. Let $i\colon V_1\subseteq V_2$ denote the inclusion. Then, for any homotopy invariant presheaf with $A$-transfers $\scrF\in\PSh_\Sigma(\Cor_k^A;\Z)$, the restriction homomorphism
$i^*\colon\scrF(V_2)\to\scrF(V_1)$
is injective.
\end{theorem}

\subsubsection{}We deduce \Cref{th:injRelA1} from the following result, which ensures the existence of a left inverse to $i^*$:

\begin{lemma}
Suppose that $V_1\subseteq V_2\subseteq \A^1_U$ are open subschemes as in \Cref{th:injRelA1}. Then there is a finite $A$-correspondence $\Phi\in \Cor^A_k(V_2, V_1)$ such that 
$[i\circ\Phi]=[\id_{V_2}]\in \overline\Cor^A_k(V_2, V_2)$.
\end{lemma}

\begin{proof}
To prove the claim we must construct a finite $A$-correspondence
$\Phi\in \Cor^A_k( V_2 , V_1 )$
along with a homotopy
$\Theta\in \Cor^A_k(\A^1\times V_2, V_2 )$
satisfying $\Theta\circ i_0=i\circ \Phi$ and $\Theta\circ i_1=\id_{V_2}$. To do this, we will make use of the following functions:
\leavevmode
\begin{center}
\begin{tabular}{@{}>{$}l<{$}>{$}l<{$}>{$}l<{$}@{}}\toprule
f\in k[\stackrel{y}{\A^1}\times \stackrel{(x,u)}{V_2}] &
h\in k[\stackrel{y}{\A^1}\times \stackrel{(x,u)}{V_2}\times\stackrel{\lambda}{\A^1}] &g\in k[\stackrel{y}{\A^1}\times \stackrel{(x,u)}{V_2}]\\
f=y^n+a_1 y^{n-1}+\cdots +a_n & h=y^n+b_1 y^{n-1}+\cdots +b_n & g=y^{n-1}+c_1 y^{n-2}+\cdots +c_{n-1}\\
\midrule
 & h\big|_{\A^1\times V_2\times0}=f & h\big|_{\A^1\times V_2\times1}=(y-x)g\\ 
f\big|_{(\A^1_U\setminus V_1)\times_U V_2 }= 1 & h\big|_{(\A^1_U\setminus V_2)\times_U V_2} =1 & g\big|_{(\A^1_U\setminus V_2)\times_U V_2}=(y-x)^{-1}\\
&& g\big|_{(V_2\setminus V_1)\times_U V_2}=1\\
&& g\big|_{Z(y-x)}=1\\\bottomrule
\end{tabular}
\end{center}
The functions $f$ and $g$ can be constructed for any $n$ big enough by using the Chinese remainder theorem \ref{ex:ChineseRemTh}. Having $f$ and $g$ we then put $h\defeq(1 -\lambda)f+\lambda (y-x)g$. 
We now aim to apply \Cref{constr:lrangleclass} to the regular functions $f$ and $h$. Keeping the notations as in \eqref{diag:pCf}, consider the following diagrams:
\[\begin{tikzcd}
\stackrel{y}{V_1}\times_U \stackrel{x}{V_2}\ar{r}{f}\ar{d}\ar{dr}{\pr^{12}_1} & \A^1 & & \stackrel{y}{V_2}\times_U \stackrel{x}{V_2}\ar{d}\ar{dr}{\pr^{22}_2}\ar{r}{(y-x)g} & \A^1 & & \stackrel{y}{V_2}\times_U \stackrel{x}{V_2}\times\A^1\ar{r}{h}\ar{dr}{\pr_2}\ar{d} & \A^1\\
\stackrel{x}{V_2} & \stackrel{y}{V_1}, & & \stackrel{x}{V_2} & \stackrel{y}{V_2}, & & \stackrel{x}{V_2}\times\A^1 & \stackrel{y}{V_2}.
\end{tikzcd}\]
Here $\pr_1^{12}, \pr^{22}_2$ and $\pr_2$ are projections. Since $f$, $(y-x)g$ and $h$ are monic polynomials in the variable $y$, it follows that $Z(f)$, $Z((y-x)g)$ and $Z(h)$ are finite over $V_2$ and $\A^1\times V_2$, respectively.
Hence \Cref{constr:lrangleclass} yields finite $A$-correspondences
\begin{align*}
\Phi^\prime &\defeq \Div^A( f )_{Z(f)}^{dy, \pr_1^{12}}\in \Cor^A_k( V_2, V_1 ),\\
\Theta^\prime &\defeq \Div^A( h )_{Z(h)}^{dy, \pr_2}\in \Cor^A_k( \A^1\times V_2, V_2 ).
\end{align*}
The properties of $f$ and $h$ above imply that 
\begin{align*}
\Theta^\prime\circ i_0&=i\circ \Phi^{\prime},\\
\Theta^\prime\circ i_1&=\Div^A( (y-x)g )^{dy,\pr_{2}^{22}}_{Z(y-x)} + \Div^A( (y-x)g )^{dy,\pr^{22}_2}_{Z(g)}.
\end{align*}
Now, according to \Cref{lm:sectZ(f)cor}, the first summand in the last equality is equal to $\langle \nu \rangle \in \Cor^A_k(V_2,V_2)$ for some invertible function $\nu$. 
Therefore, if we let $\Phi\defeq\Phi^+-\Phi^-$, where
\[
\Phi^+ \defeq\Phi'\circ\ip{\nu^{-1}},\quad
\Phi^-\defeq\Div^A( (y-x)g )^{dy,\pr^{22}_2}_{Z(g)} \circ \langle \nu^{-1} \rangle,
\]
it follows that 
\[
[\id_{V_2}] =\Div^A( (y-x)g )^{dy,\pr^{22}_2}_{Z(y-x)} \circ \langle \nu^{-1} \rangle = [i \circ \Phi]\in \ol\Cor^A_k(V_2,V_2),
\] 
as desired.
\end{proof}

\subsubsection{}We will need the following two particular cases of \Cref{th:injRelA1}:

\begin{corollary}
Suppose that $\scrF $ is a homotopy invariant presheaf with $A$-transfers over a field $k$.
Then, for any pair of open subschemes $V_1\subseteq V_2\subseteq \A^1_k$, 
the restriction homomorphism 
$\scrF (V_2)\to \scrF (V_1)$
is injective.
\end{corollary}

\begin{corollary}\label{cor:Gm-diag}
Suppose that $\scrF $ is a homotopy invariant presheaf with $A$-transfers over a field $k$, and let $U$ be an open subscheme of $\G_m\times\G_m$ such that the complement $(\G_m\times\G_m)\setminus U$ is finite over the first copy of $\G_m$.
Then the restriction homomorphism $\scrF (\G_m\times\G_m)\to \scrF (U)$ is injective. 
\end{corollary}

\section{Excision on the relative affine line}\label{section:rel-aff-line}

The aim of this section is the prove the following excision result for open subsets of a relative affine line:

\begin{theorem}\label{th:ExRelAffL}
Suppose that $U\in\Sm_k$ is an affine scheme, and let $V_1\subseteq V_2\subseteq\A^1_U$ be a pair of open subschemes such that $0_U\in V_1$. Let $i\colon V_1\subseteq V_2$ denote the inclusion.
Then, for any homotopy invariant presheaf with $A$-transfers $\scrF\in\PSh_\Sigma(\Cor_k^A;\Z)$,
the restriction homomorphism $i^*$ induces an isomorphism
\[i^*\colon \scrF (V_2\setminus 0_U)/\scrF (V_2)\xrightarrow{\cong}\scrF (V_1\setminus 0_U)/\scrF (V_1).\]
\end{theorem}

\begin{remark}
By \Cref{th:injRelA1}, the restriction maps $\scrF(V_i)\to\scrF(V_i\setminus0)$ are injective for $i=1,2$, which justifies the notation $\scrF(V_i\setminus 0_U)/\scrF(V_i)$.
\end{remark}

\subsubsection{}To prove the above theorem, we will show that $i^*$ is injective and surjective, which amounts to constructing appropriate correspondences of pairs up to homotopy. Let us first show that $i^*$ is injective:

\begin{lemma}\label{lm:AffExInj}
Suppose that $i\colon V\subseteq \A^1_U$ is an open subscheme with $0_U\in V$. Then there is a finite $A$-correspondence of pairs
$
\Phi\in \Cor^{A,\pair}_k((\A^1_U,\A^1_U\setminus 0_U), (V,V\setminus 0_U))
$ such that 
$[i\circ\Phi]=[\id_{(\A^1_U,\A^1_U\setminus 0_U)}]\in \overline\Cor^{A,\pair}_k( (\A^1_U,\A^1_U\setminus 0_U), (\A^1_U,\A^1_U\setminus 0_U) )$.
\end{lemma}

\begin{proof}
We need to construct a finite $A$-correspondence of pairs
$\Phi\in \Cor^{A,\pair}_k( (\A^1_U,\A^1_U\setminus 0_U) , (V,V\setminus 0_U) )$
along with a homotopy 
$\Theta\in \Cor^{A,\pair}_k(\A^1\times (\A^1_U,\A_U^1\setminus 0_U), (\A^1_U,\A^1_U\setminus 0_U) )$
such that $\Theta\circ i_0=i\circ \Phi$ and $\Theta\circ i_1=\id_{(\A^1_U,\A^1_U\setminus 0_U)}$.
To do this, we will make use of the following sections:
\leavevmode
\begin{center}
\begin{tabular}{@{}>{$}l<{$}>{$}l<{$}>{$}l<{$}@{}}\toprule
s\in \Gamma\p*{\stackrel{[t_0\colon t_\infty]}{ \PP^1}_{U\times\stackrel{x}{\A^1} }, \mathcal O(n)} &
\tilde s\in \Gamma\p*{\stackrel{[t_0\colon t_\infty]}{ \PP^1}_{U \times \stackrel{x}{\A^1}  \times \stackrel{\lambda}{\A^1} }, \mathcal O(n)} &
s^\prime \in \Gamma\p*{\stackrel{[t_0\colon t_\infty]}{ \PP^1}_{U\times\stackrel{x}{\A^1}  } , \mathcal O(n-1)} \\
\midrule
&\tilde s\big|_{\PP^1\times U\times\A^1\times 0}=s & \tilde s\big|_{\PP^1\times U\times\A^1\times 1}=(t_0-x t_\infty)s^\prime
\\ 
s\big|_{((\PP^1\times U) \setminus V)\times\A^1 }= t_0^n &\tilde s\big|_{\infty\times U\times\A^1\times\A^1 }= t_0^n &s^\prime\big|_{\infty\times U\times\A^1 }= t_0^{n-1}
\\
s\big|_{0\times U\times\A^1} =t_\infty^{n-1} (t_0-x t_\infty) & \tilde s\big|_{0\times U\times\A^1\times\A^1} =t_\infty^{n-1}(t_0-x t_\infty) & s^\prime\big|_{0\times U\times\A^1}=t_\infty^{n-1}\\
&& s^\prime\big|_{Z(t_0-x t_\infty)\times U}=t_\infty^{n-1}\\\bottomrule
\end{tabular}
\end{center}
Since $U$ is affine, it follows that $\calO(1)$ is ample on $\PP^1\times U \times \A^1$ and $\PP^1\times U\times \A^1\times\A^1$. Hence, for $n$ big enough,
Serre's theorem \ref{prop:CorofSerreTh} ensures the existence of the sections $s$ and $s^\prime$ as above. Having $s$ and $s'$, we then put $\tilde s\defeq(1 -\lambda)s+\lambda (t_0-x t_\infty)s^\prime$.

It follows 
by \Cref{lm:FinitnessOfVanLoc} that $Z(s)$ and $Z(\widetilde s)$ are finite over $U\times\A^1$ and $U\times\A^1\times\A^1$ respectively. Let $y\defeq t_0/t_\infty $ be the coordinate on the affine line $\A^1\subseteq \PP^1$, and consider the trivialization $dy$ of the canonical class of $\A^1$. Let moreover $p\colon\A^1\times V\to \A^1_U$ denote the composition of the projection onto $V$ followed by the inclusion $V\subseteq\A^1_U$, and let $p'\colon\A^1\times  \A^1\times U\times\A^1\to\A^1_U\times\A^1$ be the projection onto the last two coordinates. Applying \Cref{constr:lrangleclass} to the diagrams
\[\begin{tikzcd}
\stackrel{y}{V}\times \stackrel{x}{\A^1}\ar{d}[swap]{p}\ar{dr}{\pr}\ar{r}{s/t_\infty^n} & \A^1 & & \stackrel{y}{\A^1}\times U\times \stackrel{x}{\A^1}\times\stackrel{\lambda}{\A^1}\ar{d}[swap]{p'}\ar{dr}{\pr_2}\ar{r}{\wt s/t_\infty^n} & \A^1\\
\stackrel{x}{\A^1_U} & \stackrel{x}{V}, & & \stackrel{x}{\A^1_U}\times\stackrel{\lambda}{\A^1} & \stackrel{x}{\A^1_U},
\end{tikzcd}\]
we thus obtain finite $A$-correspondences  
\begin{align*}
\Phi^\prime &\defeq \Div^A( s/t_\infty^{n} )^{dy,\pr}\in \Cor^{A,\pair}_k( (\A^1_U,\A^1_U\setminus 0_U), (V,V\setminus 0_U) ),\\
\Theta^\prime &\defeq \Div^A( \tilde s/t_\infty^{n} )^{dy,\pr_2}\in \Cor^{A,\pair}_k(\A^1\times (\A^1_U,\A^1_U\setminus 0_U), (\A^1_U,\A^1_U\setminus 0_U) ).
\end{align*}
It then follows from the properties of $s$ and $\tilde s$ above that
\begin{align*}
\Theta^\prime\circ i_0&=i\circ \Phi^\prime,\\
\Theta^\prime\circ i_1&=\Div^A( (y-x)g )_{Z(y-x)} + \Div^A( (y-x)g )_{Z(g)},
\end{align*}
where  $g\defeq s^\prime/t_\infty^{n-1}\in k[\A^1\times\A^1\times U]$.
By \Cref{lm:sectZ(f)cor} the first summand in the last equality is equal to 
$\langle \nu\rangle$ for some $\nu\in k[\A^1_U]^\times$.
The second summand, $\Div^A( (y-x)g )_{Z(g)}$, is zero by \Cref{lm:ZeroPairC} since $Z(g)\cap (0\times\A^1\times U)=\varnothing$. 
Now we define $\Phi\defeq\Phi^\prime\circ \langle \nu^{-1}\rangle$ and $\Theta\defeq\Theta^\prime\circ (\langle \nu^{-1}\rangle\times\id_{\A^1})$. Then $\Theta^\prime\circ i_1=\id_{(\A^1_U,\A^1_U\setminus 0_U)}$, and the claim follows.
\end{proof}

\subsubsection{}The next step is to show surjectivity of $i^*$:

\begin{lemma}\label{lm:AffExSur}
Suppose that $i\colon V\subseteq \A^1_U$ is an open subscheme with $0_U\in V$. Then there is a finite $A$-correspondence of pairs 
$\Psi\in \Cor^{A,\pair}_k((\A^1_U,\A^1_U\setminus 0_U), (V,V\setminus 0_U))$ 
such that 
$[\Psi\circ i]=[\id_{(V,V\setminus 0_U)}]\in \overline\Cor^{A,\pair}_k( (V,V\setminus 0_U), (V,V\setminus 0_U) )$.
\end{lemma}

\begin{proof}
To prove the claim we need to construct a finite $A$-correspondence of pairs
$
\Psi\in \Cor^{A,\pair}_k( (\A^1_U,\A^1_U\setminus 0_U) , (V,V\setminus 0_U) )
$ 
along with a homotopy 
$
\Theta\in \Cor^{A,\pair}_k(\A^1\times (V,V\setminus 0_U), (V,V\setminus 0_U) )
$
such that $\Theta\circ i_0=\Psi\circ i$ and $\Theta\circ i_1=\id_{(V,V\setminus 0_U)}$. We do this via the following sections: 
\leavevmode
\begin{center}
\begin{tabular}{@{}>{$}l<{$}>{$}l<{$}>{$}l<{$}@{}}\toprule
s\in \Gamma\p*{\stackrel{[t_0\colon t_\infty]}{ \PP^1}_{U\times \stackrel{x}{\A^1}}, \mathcal O(n)} &
\tilde s\in \Gamma\p*{\stackrel{[t_0\colon t_\infty]}{ \PP^1}_{ \stackrel{x}{V} \times \stackrel{\lambda}{\A^1} }, \mathcal O(n)} &
s^\prime \in \Gamma\p*{\stackrel{[t_0\colon t_\infty]}{ \PP^1}_{ \stackrel{x}{V} } , \mathcal O(n-1)} \\
\midrule
& \tilde s\big|_{\PP^1\times V\times 0}=s & \tilde s\big|_{\PP^1\times V\times 1}=(t_0-xt_\infty)s^\prime\\ 
s\big|_{ D \times\A^1 }= t_0^n & \tilde s\big|_{ D \times V \times\A^1 }= t_0^n & g\big|_{ D \times\A^1 }= t_0^n (t_0-x t_\infty)^{-1}\\
s\big|_{0\times U\times\A^1} =t_0-x t_\infty & \tilde s\big|_{0\times V\times \A^1} =t_0-x t_\infty & s^\prime\big|_{0\times V}=t_\infty^n\\
&& s^\prime\big|_{Z(t_0-x t_\infty)}=t_\infty^{n-1}\\\bottomrule
\end{tabular}
\end{center}
Here $D\defeq(\PP^1\times U)\setminus V$ denotes the reduced closed complement,
$g\defeq s^\prime/t_\infty^{n-1}\in k[\A^1\times V]$,
and $Z(t_0-x t_\infty)\subseteq \PP^1\times V$ denotes vanishing locus of the section 
\[
t_0-xt_\infty\in \Gamma({\PP^1}\times {V},\calO(1)),
\]
with $[t_0\colon t_\infty]$ being coordinates on $\PP^1$, and $x$ the one on $V$.
Since $U$ is affine, it follows that $\mathcal O(1)$ is ample on $\PP^1\times \A^1\times U$ and $\PP^1\times \A^1\times U\times\A^1$. Hence Serre's theorem \ref{prop:CorofSerreTh} ensures the existence of the sections $s$ and $s^\prime$ as above, provided $n$ is big enough. Having $s$ and $s'$, we then put $\tilde s\defeq(1 -\lambda)s+\lambda (t_0-x t_\infty)s^\prime$. 

Next, 
it follows
by \Cref{lm:FinitnessOfVanLoc} that $Z(s)$ and $Z(\widetilde s)$ are finite over $U\times\A^1$ and $V\times\A^1$, respectively.
Let $y\defeq t_0/t_\infty$ be the coordinate on the affine line $\A^1\subseteq\PP^1$, and let us use the trivialization $dy$ of the canonical class of $\A^1$. Consider the diagrams
\[\begin{tikzcd}
\stackrel{y}{V}\times\stackrel{x}{\A^1}\ar{r}{s/t_\infty^n}\ar{dr}{\pr}\ar{d} & \A^1 & & \stackrel{y}{V}\times_U\stackrel{x}{V}\times\stackrel{\lambda}{\A^1}\ar{d}\ar{dr}{\pr'}\ar{r}{\wt s/t^n_\infty} & \A^1\\
\stackrel{x}{\A^1_U} & \stackrel{y}{V,} & & \stackrel{x}{V}\times\stackrel{\lambda}{\A^1} & \stackrel{y}{V}.
\end{tikzcd}\]
Here the map $\pr\colon\A^1\times V\to V$ is the projection, while the map $\pr'\colon\A^1\times V\times\A^1\to\A^1_U$ is the composition of the projection onto $V$ followed by the inclusion $V\subseteq\A^1_U$.
Applying \Cref{constr:lrangleclass} to these diagrams we get finite $A$-correspondences of pairs
\begin{align*}
\Psi^\prime &\defeq \Div^A( s/t_\infty^{n} )^{dy,\pr}\in \Cor^{A,\pair}_k( (\A^1_U,\A^1_U\setminus 0_U), (V,V\setminus 0_U) ),\\ 
\Theta^\prime &\defeq \Div^A( \tilde s/t_\infty^{n} )^{dy,\pr'}\in \Cor^{A,\pair}_k(\A^1\times (V,V\setminus 0_U), (V,V\setminus 0_U) ),
\end{align*}
The properties of $s$ and $s'$ above imply that 
\begin{align*}
&\Theta^\prime\circ i_0= \Psi^\prime\circ i,\\
&\Theta^\prime\circ i_1=\Div^A( (y-x)g )_{Z(y-x)} + \Div^A( (y-x)g )_{Z(g)}.
\end{align*}
By \Cref{lm:sectZ(f)cor}, the first summand in the last equality is equal to 
$\langle \nu\rangle\in \Cor^{A,\pair}_k((V,V\setminus 0_U),(V,V\setminus 0_U))$
for some $\nu\in k[V]^\times$. The second summand is zero by \Cref{lm:ZeroPairC}, since $Z(g)\cap (0\times V)=\varnothing$. 
Hence the $A$-correspondences
$\Psi\defeq \langle \nu^{-1}\rangle \circ \Psi^\prime$ and
$\Theta\defeq\langle \nu^{-1}\rangle \circ \Theta^\prime$ have the desired properties.
\end{proof}

\begin{proof}[Proof of \Cref{th:ExRelAffL}]
\Cref{lm:AffExInj} and \Cref{lm:AffExSur} immediately imply the claim for the case of $V_2=\A^1_U$.
In general, it follows that we have natural isomorphisms
\[\scrF (V_2\setminus 0_U)/\scrF (V_2)\cong
\scrF (\A^1_U\setminus 0_U)/\scrF (\A^1_U) \cong
\scrF (V_1\setminus 0_U)/\scrF (V_1),\]
which shows the claim.
\end{proof}

\subsubsection{}Arguing similarly as in the proof of \Cref{th:ExRelAffL}, we obtain also an excision result for a nonrelative affine line:

\begin{theorem}\label{th:ExAffL}
Consider the function field $K\defeq k(U)$ of some integral scheme $U\in \Sm_k$. Let $z$ be a closed point in $\A^1_K$, and let $i\colon V_1\subseteq V_2$ be an inclusion of two open subschemes of $\A^1_K$ such that $z\in V_1$.
Then, for any homotopy invariant presheaf with $A$-transfers $\scrF\in\PSh_\Sigma(\Cor_k^A;\Z)$,
the restriction homomorphism $i^*$ induces an isomorphism
\[i^*\colon \scrF (V_2\setminus z)/\scrF (V_2)\xrightarrow{\cong}\scrF (V_1\setminus z)/\scrF (V_1).\]
\end{theorem}

\begin{proof}
The proof is parallel to the proof of \Cref{th:ExRelAffL}. All we need to do is to replace 
the line bundle $\calO(1)$ by $\calO(d)$, where $d\defeq\deg_K k(z)$; the section $t_0\in \Gamma(\PP^1_{\A^1_K},\calO(1))$ 
by a section $\nu\in \Gamma(\PP^1_{\A^1_K},\calO(d))$ such that $Z(\nu)=z\times{\A^1_K}$;
and the section $t_\infty$ by $t_\infty^d$. 
\end{proof}

\section{Injectivity for semilocal schemes}\label{section:inj-loc-sch}

In this section we will assume that the base field $k$ is infinite.

\begin{theorem}\label{thm:inj-loc-sch}
Let $X$ be a smooth $k$-scheme and let $x_1,\dots,x_r\in X$ be finitely many closed points. Let $U\defeq\Spec\calO_{X,x_1,\dots,x_r}$ and write $j\colon U\rightarrow X$ for the canonical inclusion. Let $Z\hookrightarrow X$ be a closed subscheme with $x_1,\dots,x_r\in Z$, and let $i\colon U\setminus Z\to U$ be the immersion of the open complement to the semilocalization of $Z$ at the points $x_1,\dots,x_r$. 
Then, for any homotopy invariant presheaf with $A$-transfers $\scrF\in\PSh_\Sigma(\Cor_k^A;\Z)$, the homomorphism $i^*\colon \scrF(U)\to \scrF(U\setminus Z)$ is injective.
\end{theorem}

\subsubsection{}\Cref{thm:inj-loc-sch} is an immediate consequence of the following moving lemma:

\begin{lemma}\label{lemma:inj-loc-sch}
Assume the hypotheses of \Cref{thm:inj-loc-sch}.
Then there exists a finite $A$-correspondence $\Phi\in\Cor^A_k(U,X\setminus Z)$ such that the diagram
\[\begin{tikzcd}
& X\setminus Z\ar[d,"i"]\\
U\ar[ur,"\Phi"]\ar[r,"j"] & X
\end{tikzcd}\]
commutes up to homotopy.
\end{lemma}

\subsubsection{}We prove \Cref{lemma:inj-loc-sch} by constructing an appropriate relative curve $\calC$ over $U$ along with a good compactification $\ol\calC$ of $\calC$. The desired finite $A$-correspondence will then be defined by using certain sections on $\ol\calC$. 

\begin{lemma}\label{lemma:relcurve1}
Assume the hypotheses of \Cref{thm:inj-loc-sch}. Then there exists a diagram 
\[
X\xleftarrow{v}\calC\xrightarrow{j}\ol\calC\xrightarrow{p}U
\] 
in $\EssSm_k$, satisfying the following properties:
\begin{itemize}
\item[(1)] $p\colon\ol\calC\to U$ is a relative projective curve, $j\colon\calC\to\ol\calC$ is an open immersion, and the composition $p\circ j$ is smooth.  
\item[(2)] The map $p\circ j$ admits a section $\Delta\colon U\to \calC$. By abuse of notation, we write $\Delta$ also for the image of the morphism $\Delta$. 
\item[(3)] Let $\calZ\defeq v^{-1}(Z)\subseteq\calC$. Then $\calZ$ is finite over $U$. 
\item[(4)] $D\defeq\ol\calC\setminus\calC$ is finite over $U$.
\item[(5)] The relative curve $\ol\calC$ has an ample line bundle $\calO_{\ol\calC}(1)$.
\item[(6)] There is a trivialization $\mu\colon\calO_\calC\xrightarrow{\cong}\omega_{\calC/U}$. 
\end{itemize}
\end{lemma}

\begin{proof}
We apply \Cref{lm:RelCurves} with $\pi=\id\colon X\to X$.
\end{proof}

\begin{proof}[Proof of \Cref{lemma:inj-loc-sch}.]
First of all we apply \Cref{lemma:relcurve1}. Then it follows from Serre's theorem \ref{prop:CorofSerreTh} that there is an integer $l\gg0$ and a section $d\in\Gamma(\ol\calC,\calO(l))$ such that $D\subseteq Z(d)$, $Z(d)\cap \calZ=\varnothing$ and $Z(d)$ is finite over $U$. For notational simplicity, let us redenote $\calO(l)$ by $\calO(1)$, and redenote $D\defeq Z(d)$. Now our aim is to construct the following sections:
\begin{center}
\begin{tabular}{@{}>{$}l<{$}>{$}l<{$}>{$}l<{$}>{$}l<{$}@{}}\toprule
s\in \Gamma(\ol\calC,\calO(n) ) &
\wt s\in \Gamma(\ol\calC\times\stackrel{\lambda}{\A^1},\calO(n) ) &
s^\prime \in \Gamma(\ol\calC,\calO(n)\otimes\scrL(\Delta)^{-1}) &
\delta\in\Gamma(\ol\calC,\scrL(\Delta))  \\
\midrule
Z(s|_{\calZ\amalg D})=\varnothing & \wt s|_{\ol\calC\times 0}=s & Z(s'|_{\calZ\amalg D\amalg\Delta})=\varnothing & Z(\delta)=\Delta \\
 & \wt s|_{\ol\calC\times1}=s'\otimes\delta &  & \\
& \wt s|_{D\times\A^1}=s & &\\
\bottomrule
\end{tabular}
\end{center}
To do this, let $\delta$ be a section of $\scrL(\Delta)$ with $Z(\delta)=\Delta$, and choose, using \Cref{prop:CorofSerreTh}, an integer $n\gg0$ such that the restriction maps
\begin{align*}
\Gamma(\ol\calC,\calO(n)\otimes\scrL(\Delta)^{-1}) &\to\Gamma(\calZ\amalg D\amalg\Delta,\calO(n)\otimes\scrL(\Delta)^{-1}),\\
\Gamma(\ol\calC,\calO(n))&\to\Gamma(\calZ\amalg D,\calO(n))
\end{align*}
are surjective. We can then find a global section $s'$ of $\calO(n)\otimes\scrL(\Delta)^{-1}$ such that $s'|_{\calZ\amalg D\amalg\Delta}$ is invertible. Let $s$ be a lift of $s'\delta|_{\calZ\amalg D}\in\Gamma(\calZ\amalg D,\calO(n))$, and define $\wt s\defeq(1-\lambda)s+\lambda s'\otimes\delta$. 
We now aim to apply \Cref{constr:lrangleclass} to the diagrams
\[\begin{tikzcd}
\calC\ar[rr,shift left,"s'\otimes\delta/d^n"]\ar[rr,shift right,swap,"s/d^n"]\ar{d}[swap]{p\circ j}\ar{drr}[swap]{v} & & \A^1 & & \calC\times\A^1\ar{d}[swap]{(p\circ j)\times\A^1}\ar{r}{\wt s/d^n}\ar{dr}{v\circ\pr} & \A^1\\
U & & X, & & U\times\A^1 & X.
\end{tikzcd}\]
Here $\pr\colon\calC\times\A^1\to\calC$ is the projection. By \Cref{lm:FinitnessOfVanLoc}, the vanishing loci $Z(s)$ and $Z(\widetilde s)$ 
are finite over $U$ and $U\times\A^1$, respectively.
Hence we obtain finite $A$-correspondences
\begin{align*}
\Phi'&\defeq \Div^A( s/d^n )_{Z(s)}^{\mu,v}-\Div^A( s'\otimes\delta/d^n )_{Z(s')}^{\mu,v}\in\Cor_k^A(U,X\setminus Z),\\
\Theta'&\defeq \Div^A( \tilde s/d^n )_{Z(\wt s)}^{\mu,v\circ \pr^{\calC\times\A^1}_{\calC}}-\Div^A( s'\otimes\delta/d^n )_{Z(s')}^{\mu,v}\circ\pr_U^{U\times\A^1}\in\Cor_k^A(U\times\A^1,X).
\end{align*}
Then the properties of the sections above imply that
$
\Theta^\prime\circ i_0=i\circ \Phi^\prime
$,
and
\Cref{lm:sectZ(f)cor} implies that
$ 
\Theta^\prime\circ i_1=
j\circ \langle\nu\rangle
$
for some $\nu\in k[U]^\times$.
Now let
$\Phi\defeq\Phi^\prime\circ  \langle\nu^{-1}\rangle$.
Then $\Theta\defeq \Theta'\circ  \langle\nu^{-1}\rangle$
gives the required homotopy, satisfying
$\Theta\circ i_0=i\circ \Phi$ and
$j=\Theta\circ i_1$.
\end{proof}

\section{Étale excision}\label{section:et-exc}

In this section we assume that the base field is infinite. The main result of the section is the following étale excision result for homotopy invariant presheaves with $A$-transfers:

\begin{theorem}\label{th:etex}
Let $X\in\Sm_k$ and suppose that $\pi\colon (X^\prime, Z^\prime)\to (X,Z)$ is an étale neighborhood of $Z$ in $X$. Assume also that $z\in Z$ and $z^\prime\in Z^\prime$ are two closed points such that $\pi(z^\prime)=z$. Write $U\defeq X_z=\Spec\calO_{X,z}$ for the corresponding local scheme, and similarly $U'\defeq X'_{z'}$. 
Then, for any homotopy invariant presheaf with $A$-transfers $\scrF\in\PSh_\Sigma(\Cor_k^A;\Z)$, the map $\pi^*$ induces an isomorphism 
\[
\pi^*\colon\scrF (X_z\setminus Z_z)/\scrF (X_z)\xrightarrow{\cong} \scrF (X_{z'}^\prime\setminus Z_{z'}^\prime)/\scrF (X_{z'}^\prime).
\] 
\end{theorem}

%
\newcommand{\ovcCpp}{\overline{\cCpp}}
\newcommand{\cCpp}{ {\mathcal C^{\prime\prime}} }
\newcommand{\ovcCp}{\overline{\cCp}}
\newcommand{\cCp}{ {\mathcal C^\prime} }
\newcommand{\ovcC}{\overline{\cC}}
\newcommand{\cC}{\mathcal C }
\newcommand{\cZ}{\mathcal Z}
\newcommand{\cZp}{\cZ^\prime}
\newcommand{\cZpp}{\cZ^{\prime\prime}}
\newcommand{\Upri}{U^\prime}
\newcommand{\Dpri}{D^\prime}
\newcommand{\Dppri}{D^{\prime\prime}}
\newcommand{\cCinf}{\cC_\infty}
\newcommand{\cCinfp}{\cC^\prime_\infty}
\newcommand{\cCinfpp}{\cCpp_\infty}
\newcommand{\ovC}{\overline{C}}
\newcommand{\ovCp}{\overline{C^\prime}}
\newcommand{\Xpri}{X^\prime}
\newcommand{\ovX}{\overline X}
\newcommand{\ovXp}{\overline{X^\prime}}
\newcommand{\Xinf}{X_\infty}
\newcommand{\Xinfp}{X_\infty^\prime}
\newcommand{\Zpri}{Z^\prime}
\newcommand{\ovZ}{\overline Z}
\newcommand{\ovZp}{\overline{Z^\prime}}
\newcommand{\Zinf}{Z_\infty}
\newcommand{\Zinfp}{Z_\infty^\prime}
\newcommand{\ovpi}{\overline{\pi}}

\subsubsection{}
The proof of \Cref{th:etex} relies on some geometric input. Our main tool for this is \Cref{lm:RelCurves}; we refer the reader to the appendix for details around this construction. 

Having \Cref{lm:RelCurves} at hand, we start out by showing that the map $\pi^*$ is injective:

\begin{lemma}\label{lm:injACoretex}
Under the assumptions of \Cref{th:etex}
there is a finite $A$-correspondence $\Phi\in\Cor^A_k(U,X^\prime)$ satisfying $\pi\circ \Phi\sim_{\A^1}i$, where $i\colon U\to X$ denotes canonical embedding.
\end{lemma}

\begin{proof}
Applying \Cref{lm:RelCurves} we obtain a morphism of relative curves $\varpi\colon \calC^\prime\to \calC$ over $U$, with compactification $\overline{\varpi}\colon \ovcCp\to\ovcC$, and subschemes $D,\Delta,\calZ\subseteq\ovcC$, $\Dpri,\Delta_Z^\prime,\cZp\subseteq\ovcCp$ as in \Cref{lm:RelCurves}. 
Let $\delta\in \Gamma(\ovcC,\scrL(\Delta))$ be a section such that $Z(\delta)=\Delta$.
Our first aim is to prove that there is an integer $N$ such that for all $n\ge N$, 
there exist 
sections satisfying the following conditions:
\leavevmode
\begin{center}
\begin{tabular}{@{}>{$}l<{$}>{$}l<{$}>{$}l<{$}@{}}\toprule
s\in \Gamma( \ovcC , \mathcal O(n) ) &
\tilde s\in \Gamma( \ovcC\times \stackrel{\lambda}{\A^1} , \mathcal O(n) ) &
s^\prime \in \Gamma( \ovcC , \mathcal O(n)\otimes\scrL(\Delta)^{-1} ) \\
\midrule
& \tilde s\big|_{\ovcC\times 0}=s & \tilde s\big|_{\ovcC\times 1}=\delta \otimes s^\prime\\ 
Z(s\big|_{D})= \varnothing
& Z(\tilde s\big|_{D\times\A^1})= \pr^*(s) &\\ 
s\big|_{\mathcal Z} =\delta \otimes s^\prime & \tilde s\big|_{\mathcal Z\times\A^1} =\delta \otimes s^\prime  & Z(s^\prime\big|_{\mathcal Z})=\varnothing\\
& Z(\tilde s)\cap Z(d)=\varnothing &\\\bottomrule
\end{tabular}
\end{center}
In addition, we will require that $Z(s)=Z_0\amalg Z^\prime_0$ and that there exists
a regular map $l\colon Z_0\to \cCp$ satisfying $\varpi \circ l=\id_{Z_0}$.
Here $\pr\colon \ovcC\times\A^1\to \ovcC$ is the canonical projection.

To do this we start the following preparations.
Let $\calO_{\cCp}(1)\defeq \overline{\varpi}^*(\calO(1))$. Then, since $\overline{\varpi}$ is finite, $\calO_{\cCp}(1)$ is an ample bundle on $\ovcCp$.
Since $\varpi$ induces isomorphisms $\cZp\cong\calZ$ and $\Delta_Z^\prime\cong \Delta\times_{\cC} \calZ$, 
there is a section $\delta^\prime\in \Gamma(\cZp,\scrL^\prime)$ such that $Z(\delta^\prime)=\Delta_Z^\prime$ for some line bundle $\scrL^\prime$ on $\cZp$.
Since $\cZp$ is a finite scheme over a local scheme $U$, $\cZp$ is semilocal and any line bundle on $\cZp$ is trivial.
Hence there is an isomorphism $\scrL^\prime\cong \calO_{\cCp}(m)\big|_{\calZ'}$ for any $m\in \Z$.
Similarly, since the subscheme $\Dpri\subseteq\ovcCp$ is finite over $U$, for any $m\in \Z$, the line bundle $\calO_{\cCp}(m)\big|_{\Dpri}$ is trivial.
Now, applying \Cref{lm:clEnbSect}
to the morphism $\overline\varpi\colon \ovcCp\to \ovcC$ and the subschemes $\Dpri$ and $\calZ$
we construct, 
for some $m\in \Z$, a section $\xi\in \Gamma(\cCp,\calO_{\calC^\prime}(m))$ 
such that there is a closed embedding $Z(\xi)\to \calC$, and such that
$Z(\xi\big|_{\overline\varpi^{-1}(\calZ)})=\Delta_Z^\prime$. 
Define $Z_0\defeq{\overline\varpi}(Z(\xi))\subseteq \calC\subseteq\overline \calC$ and put $\scrL\defeq\scrL(Z)$. Let $\zeta\in \Gamma(\overline \calC,\scrL)$ be a section with $Z(\zeta)=Z$. Then $Z(\zeta\big|_{\calZ})=\Delta_Z$.

Using Serre's theorem \ref{prop:CorofSerreTh} we can choose an integer $N\in \Z$ such that for all $n\ge N$, the restriction homomorphisms 
\begin{align*}\Gamma(\overline\calC,\calO(n)\otimes\scrL^{-1})&\to \Gamma(\calZ\amalg D,\calO(n)\otimes\scrL^{-1})\\
\Gamma(\overline\calC,\calO(n))& \to \Gamma( (\calZ\cup \Delta) \amalg D,\calO(n))
\end{align*}
are surjective.
Then, since $\calZ\amalg D$ is semilocal, 
there is a section $\zeta^\prime\in \Gamma(\ol\calC, \calO(n)\otimes\scrL^{-1})$ such that $\zeta^\prime\big|_{\calZ\amalg D}$ is invertible.
Define $s\defeq\zeta\otimes \zeta^\prime\in \Gamma(\overline\calC, \calO(n))$.

Now choose a section $s_1\in \Gamma(\overline\calC, \calO(n))$ such that $s_1\big|_{\Delta}=0$ and $s_1\big|_{\calZ} =s$. We then put $\tilde s\defeq(1-\lambda)s+\lambda s_1$.
Since $s_1\big|_{\Delta}=0$, there is a section $s^\prime \in \Gamma(\overline\calC, \calO(n)\otimes \scrL(\Delta)^{-1})$ such that $s_1=\delta \otimes s^\prime$, where $\delta\in \Gamma(\overline\calC, \scrL(\Delta))$ satisfies $Z(\delta)=\Delta$.
Moreover, since by construction $Z(s_1\big|_{\calZ})=\Delta_Z=Z(\delta\big|_{\calZ})$, it follows that $s^\prime\big|_{\calZ}$ is invertible and so $Z(s^\prime\big|_{\calZ})=\varnothing$.
Hence the desired sections $s$, $\tilde s$, and $s^\prime$ are constructed.
Moreover it follows by \Cref{lm:FinitnessOfVanLoc} now that $Z(s)$ and $Z(\widetilde s)$ are finite over $U$ and $U\times\A^1$ respectively.

By construction, the morphism $\varpi$ induces an isomorphism between the closed subschemes $l(Z_0)\subseteq \calC^\prime$ and $Z_0$. Since $\varpi$ is étale,
it follows that 
$\varpi^{-1}(Z_0)=l(Z_0)\amalg \wh Z_0$. Hence we can define an étale neighborhood 
$\varpi^+\colon (\calC^\prime\setminus\wh Z_0,l(Z_0))\to (\calC,Z_0)$ such that $\varpi^+(Z_0)=l(Z_0)$. Consider the diagrams
\[\begin{tikzcd}
\calC'\ar{d}[swap]{p\circ j\circ\varpi}\ar{r}{\varpi^*(s/d^n)}\ar{dr}{v'} & \A^1 & & \calC\times\A^1\ar{d}[swap]{(p\circ j)\times\A^1}\ar{r}{\wt s/d^n}\ar{dr}{v\circ\pr} & \A^1\\
U & X', & & U\times\A^1 & X,
\end{tikzcd}\]
where $\pr\colon \calC\times\A^1\to \calC$ is the projection.
Applying \Cref{constr:lrangleclass} to these diagrams we obtain finite $A$-correspondences
\begin{align*}
&\Phi^\prime\defeq \Div^A( \varpi^*(s/d^n) )_{Z_0}^{\varpi^*(\mu), v^\prime}\in \Cor^{A,\pair}_k((U,U\setminus Z\times_X U), (X^\prime, X^\prime\setminus Z^\prime)),\\
&\Theta^\prime \defeq \Div^A( \tilde s/d^n )^{v\circ \pr}\in \Cor^{A,\pair}_k(\A^1\times(U,U\setminus Z\times_X U), (X, X\setminus Z)).
\end{align*}
It follows from the list of properties above, \Cref{lm:sectZ(f)cor}, and \Cref{lm:ZeroPairC}
that 
$\Theta^\prime\circ i_1= i\circ\langle \nu\rangle$ 
for some invertible function $\nu\in k[U]^\times$.
If we let $\Phi\defeq\Phi^\prime\circ \langle \nu^{-1}\rangle$ and $\Theta\defeq \Theta^\prime\circ \langle \nu^{-1}\rangle$, it follows that
$\Theta\circ i_1= i$. So to prove the lemma it is enough to show that
$\Theta\circ i_0 \sim_{\A^1} \pi\circ \Phi$.

Since $\ol\varpi$ is finite, it is affine. Hence for some Zariski neighborhood $V^\prime$ of $l(Z_0)$ in $\calC^\prime\setminus \wh Z_0$, the restriction 
$\varpi\big|_{V}$ is affine. Then, for some Zariski neighborhood $V$ of $Z_0$ in $\calC$, there is 
a closed embedding $c\colon V^{\prime\prime}\subseteq \A^r\times V$, where $V''\defeq V^\prime\cap \varpi^{-1}(V)$, which is such that $c(l(Z_0))= 0\times Z_0$.
Let $f_1,\dots, f_r\in k[\A^r\times V]$ be functions satisfying 
$f_i\big|_{c(V^{\prime\prime})}=0$ and $f_i\big|_{\A^r\times Z_0}=x_i$,
where the $x_i$'s denote the coordinate functions on $\A^r$.
For $i=1,\dots,r$, let $\wt f_i\defeq(1-\lambda)f_i+\lambda x_i$ and consider the closed subscheme $Z(\wt f_1,\dots, \wt f_r)\subseteq \A^r\times V\times \A^1$.
Then the projection $\pr\colon Z(\tilde f_1,\dots, \tilde f_r)\to V\times \A^1$ is étale over $Z_0\times\A^1$.
Let $W\subseteq Z(\tilde f_1,\dots,\tilde f_r)$ be a Zariski neighborhood of $0\times Z_0\times\A^1$ such that the restriction of the projection $\pr_W\colon W\to V\times\A^1$ is étale. Furthermore, let $t$ be the pullback of $s/d^n$ from $V$ to $W$, and let $i_V\colon V\to \calC$ denote the open embedding. Applying \Cref{constr:lrangleclass} to the diagram
\[\begin{tikzcd}
W\ar{r}{t}\ar{d}[swap]{v\times\A^1}\ar{dr}{v\circ i_V\circ\pr_W} & \A^1\\
U\times\A^1 & X,
\end{tikzcd}\]
we obtain a homotopy 
\[
\Div^A( t)_{0\times Z_0\times\A^1}^{\pr_W^*(\omega),v\circ i_V\circ \pr_W}\in \Cor^{A,\pair}_k(\A^1\times(U,U\setminus Z\times_X U), (X,X\setminus Z))
\]
connecting  
$\pi\circ \Phi^\prime = \Div^A( \varpi^*(s/d^n) )^{\varpi^*(\omega),v\circ \varpi}$
and
$\Theta\circ i_0 = \Div^A( s/d^n )^{\omega,v}$.
\end{proof}

\subsubsection{}Before we move on to the surjective part of étale excision, we need the following lemma:

\begin{lemma}\label{lm:locup}
Suppose that $\Char k\ne2$, and let $X\in \Sm_k$. Let $Z\subseteq X$ be a closed subscheme and $z\in X$ a closed point.
Write $U$ for the essentially smooth local scheme $U\defeq X^h_z=\Spec\calO^h_{X,z}$, and let
$\lambda\in k[U]^\times$ be an invertible regular function satisfying $\lambda\big|_{Z\times_X U}=1$.
Then 
\[i\circ \langle \lambda \rangle\sim_{\A^1} i\in \Cor^{A,\pair}_k( (U, U\setminus Z\times_X U), (X, X\setminus Z) ),\]
where $i$ denotes the canonical morphism $i\colon U\to X$.
\end{lemma}

\begin{proof}
Lift $\lambda$ to an invertible section on some affine Zariski neighborhood $V\subseteq X$ of the point $z\in X$.
Then $\lambda\big|_{Z\times_X V^\prime}=1$ for some other Zariski neighborhood $V^\prime\subseteq V$ of $z$;
shrinking $X$ to $V^\prime$ we may assume that $\lambda\in k[X]^\times$ with $\lambda\big|_Z=1$.

Consider the étale covering $\pi\colon X^\prime\to X$, where $X^\prime=\Spec k[X][w]/(w^2-\lambda)$.
Let $Z^\prime$ be the closed subscheme of $X^\prime$ given by $Z^\prime\defeq \Spec k[Z][w]/(w-1)$, so that $Z'\cong Z$.
Then $(X^\prime, Z^\prime)\to (X,Z)$ is an étale neighborhood. By \Cref{lm:injACoretex}
there exists a finite $A$-correspondence of pairs
\[\Phi\in \Cor^{A,\pair}_k((U,U\setminus Z\times_X U) , (X^\prime,X^\prime\setminus Z^\prime) )
\]
such that
$\pi \circ \Phi\sim_{\A^1} i$ in $\Cor^{A,\pair}_k( (U, U\setminus Z\times_X U), (X, X\setminus Z) )$.
On other hand, \Cref{lm:sqroot} implies that 
\[\langle \lambda\rangle \circ \pi = \pi \circ \langle \pi^*(\lambda)\rangle = \pi\circ \langle w^2\rangle \sim_{\A^1} \pi\in \Cor^{A,\pair}_k((X,X\setminus Z), (X,X\setminus Z)).\]
Hence $ i \circ \langle i^*(\lambda) \rangle = \langle \lambda \rangle\circ i\sim_{\A^1} \langle \lambda \rangle\circ \pi \circ \Phi \sim_{\A^1} \pi\circ \Phi \sim_{\A^1} i$.
\end{proof}

\begin{lemma}\label{lemma:etex-surj}
Let $i^\prime\colon U^\prime=X^\prime_{z^\prime}\to X^\prime$ denote the canonical embedding. Then under the assumptions of \Cref{th:etex},
there exists $\Phi\in \Cor^A_k(U,X^\prime)$ such that $\Phi\circ\pi\sim_{\A^1}i^\prime$. 
\end{lemma}
\begin{proof}
\newcommand{\pprime}{ {\prime\prime} }

Using \Cref{lm:RelCurves} we construct 
relative projective curves 
$p^\prime\colon \ol\calC^\prime\to U$, $p^\pprime\colon \ol\calC^\pprime\to U^\prime$, 
along with the other data related to the first two rows of the diagram \eqref{diag:relCurves}.

Since $U^\prime$ is essentially smooth, we have $\Delta^{\prime\prime}\cong U'$. Moreover, since $p^\pprime\colon \calC^\pprime\to U^\pprime$ is a smooth morphism with fibers of dimension one, it follows that $\Delta^\pprime$ is a smooth divisor on $\calC^\pprime$. Hence it is a smooth divisor on $\ol\calC^\pprime$ as well and there is an invertible bundle $\scrL(\Delta^{\prime\prime})$ on $\ol\calC^\pprime$ and a section $\delta\in \Gamma(\ol\calC^\pprime, \scrL(\Delta^{\prime\prime}))$ such that $Z(\delta)=\Delta^\pprime$.  

Since $\calZ^\prime$ is finite over the local scheme $U$, $\calZ^\prime$ is semilocal.
Let $\delta^\prime\in k[\calZ^\prime]$ be a regular function such that $\delta^\prime\big|_{\Delta_Z^\prime}=0$, and such that $\delta'$ is invertible on the closed points of $\calZ^\prime$ outside $\Delta_Z'$. 
Then the closed fibers of $Z(\delta^\prime)$ and $\Delta_Z^\prime$ coincide. 
Now $Z(\delta^\prime)$ is finite over $U$ since it is a closed subset in $\calZ^\prime$. Moreover,
$\Delta_Z$ is finite over $U$ since $\Delta_Z$ is isomorphic to the closed subscheme $U\times_X Z$ in $U$.
Hence $Z(\delta^\prime)=\Delta_Z$ by Nakayama's lemma.

Using the notations of \Cref{lm:RelCurves}, define $\calO_{\ol\calC^\prime}(1)\defeq\ol\varpi'^*(\calO(1))$ and $\calO_{\ol\calC^\pprime}(1)\defeq \ol\varpi^*\ol\varpi'^*(\calO(1))$. Then, since $\calO(1)$ is ample and $\ol\varpi,\ol\varpi^\prime$ are finite, it follows that $\calO_{\ol\calC^\prime}(1)$ and $\calO_{\ol\calC^\pprime}(1)$ are ample.
Serre's theorem \ref{prop:CorofSerreTh} then tells us that there is an integer $n\in  \Z$ such that the restriction homomorphisms
\begin{align}
\Gamma( \ol\calC^\prime , \calO(n) )&\to 
  \Gamma( \calZ^\pprime \amalg D^\pprime , \calO(n)\otimes\scrL(\Delta^\pprime) ),\label{eq:surrestrSutEtEx}\\
\Gamma( \ol\calC^\pprime , \calO(n)\otimes\scrL(\Delta^\pprime) )&\to 
  \Gamma( \calZ^\pprime \amalg D^\pprime , \calO(n)\otimes\scrL(\Delta^\pprime) )\label{eq:surrestrSutEtEx2}
\end{align}
are surjective.
As mentioned above, $\calZ$ and $D$ are finite over $U$, so it follows that $\calZ^\prime$ and $D^\prime$ are semilocal, and moreover that
there are trivializations 
$\xi_Z\colon \calO_{\calZ^\prime}\stackrel{\cong}{\to} \calO_{\ovcCp}(1)\big|_{\calZ^\prime}$ and
$\xi_D\colon \calO_{D^\prime}\stackrel{\cong}{\to} \calO_{\ovcCp}(1)\big|_{D^\prime}$.
Now using surjectivity of the map \eqref{eq:surrestrSutEtEx} we find a section 
\[
s\in \Gamma( \ovcCp , \calO(n) ),\quad s\big|_{\calZ^\prime}=\delta\otimes\xi_Z^{\otimes n},\quad s\big|_{D^\prime}=\xi_D^{\otimes n}.
\]

By the same reason as above there is some trivialization $\xi_Z^\prime\colon \calO_{\calZ^\pprime}\stackrel{\cong}{\to}\scrL(\Delta^\pprime)\big|_{\calZ^\pprime}$.
Then $b_1=\varpi^*(\delta^\prime)$ and $b_2=\delta\otimes {\xi_Z^\prime}^{-1}$ 
are two regular functions on $\calZ^\pprime$ such that 
$Z( b_1 )=Z( b_2 )=\Delta_Z^\pprime$. 
Hence there is an invertible function $\nu\in k[\calZ^\pprime]^\times$ such that 
$
\varpi^*(\delta^\prime)\nu = \delta\otimes {\xi_Z^\prime}^{-1}.
$
Indeed, $\nu$ is uniquely defined by the equality $b_1\nu = b_2$ on the closed subscheme $Z(I)\subseteq \calZ^\pprime$. Here $I\defeq\ker(m^{b_1})$, where $m^{b_1}\in \End(k[\calZ^\pprime])$ is defined as multiplication by $b_1$. Moreover, the equality $b_1\nu = b_2$ implies that $\nu$ is invertible on $Z(I)$, and any lift of $\nu$ to a regular function on $\calZ^\pprime$ satisfies the equality $b_1\nu = b_2$ as well. So it is enough to choose a lift such that $\nu$ is nonzero at the closed points of $\calZ^\pprime\setminus Z(I)$. 

Using surjectivity of the second map \eqref{eq:surrestrSutEtEx2}, we find a section
\[s^\prime\in \Gamma( \ovcCpp , \calO(n)\otimes\scrL(\Delta^\pprime)^{-1} ),\quad
s^\prime\big|_{\calZ^\pprime}={\ol\varpi^\prime}^*(\xi_Z)^{\otimes n}\nu,\quad s^\prime\big|_{D^\pprime}=\ol\varpi^*(\xi_D^{\otimes n})\otimes\delta\big|_{D^\pprime}^{-1}.
\]
Note that the section $\delta\big|_{D^\pprime}^{-1}$ is well defined since $\Delta^\pprime\cap D^\pprime=\varnothing$. 
Now define $\tilde{s}\defeq (1-\lambda)s+\lambda s^\prime$. Then we have:
\leavevmode
\begin{center}
\begin{tabular}{@{}>{$}l<{$}>{$}l<{$}>{$}l<{$}@{}}\toprule
s\in \Gamma( \ovcCp , \mathcal O(n) ) &
\tilde s\in \Gamma( \ovcCpp\times \stackrel{\lambda}{\A^1} , \mathcal O(n) ) &
s^\prime \in \Gamma( \ovcCpp , \mathcal O(n)\otimes\scrL(\Delta^\prime)^{-1} ) \\
\midrule
& \tilde s\big|_{\ovcCpp\times 0}={\varpi^\prime}^*(s) & \tilde s\big|_{\ovcCpp\times 1}=\delta \otimes s^\prime
\\ 
Z(s\big|_{D^\prime})= \varnothing
& Z(\tilde s\big|_{D^{\prime\prime} })= \pr^*({\varpi^\prime}^*(s)) &
\\ 
s\big|_{\mathcal Z^\prime\times_U Z} =\delta^\prime \otimes s^\prime & \tilde s\big|_{\mathcal Z^{\prime\prime}\times\A^1} =\delta \otimes s^\prime  & Z(s^\prime\big|_{\mathcal Z}^{\prime\prime})=\varnothing
\\\bottomrule
\end{tabular}
\end{center}
We now aim to apply \Cref{constr:lrangleclass} to the diagrams
\[\begin{tikzcd}
\calC'\ar{d}[swap]{p'\circ j'}\ar{dr}{v'}\ar{r}{s/d^n} & \A^1 & & \calC''\times\A^1\ar{d}[swap]{(p''\circ j'')\times\A^1}\ar{dr}{v''\circ\pr}\ar{r}{\wt s/d^n} & \A^1\\
U & X', & & U'\times\A^1 & X'.
\end{tikzcd}\]
Here $\pr\colon\calC''\times\A^1\to\calC''$ is the projection. 
By \Cref{lm:FinitnessOfVanLoc}, $Z(s)$ and $Z(\widetilde s)$ are finite over $U$ and $U^\prime\times\A^1$, respectively.
Hence \Cref{constr:lrangleclass} yields finite $A$-correspondences
\begin{align*}
\Phi^\prime&\defeq \Div^A( s/d^n )^{\mu^\prime,v^\prime}\in \Cor^A_k(U,X^\prime),\\
\Theta^\prime&\defeq \Div^A( \tilde s/d^n )^{\varpi^*(\mu^\prime),v^\pprime\circ\pr}\in \Cor^A_k(U^\prime\times\A^1,X^\prime).
\end{align*}
Then, by construction,
\begin{align*}
&\Theta^\prime\circ i_0=\Phi^\prime\circ\pi,\\
&\Theta^\prime\circ i_1=\Div^A( \delta\otimes s^\prime/d^n)^{\varpi^*(\mu^\prime),v^\pprime}_{\Delta^\pprime} + \Div^A(\delta\otimes s^\prime/d^n)^{\varpi^*(\mu^\prime),v^\pprime}_{Z(s^\prime)}.
\end{align*}
By \Cref{lm:ZeroPairC} we have
\[
\Div^A(\delta\otimes s^\prime/d^n)^{\varpi^*(\mu^\prime),v^\pprime}_{Z(s^\prime)}=0\in \Cor^{A,\pair}_k( (U^\prime,U^\prime\setminus Z^\prime\times_{X^\prime} U^\prime), (X^\prime,X^\prime\setminus Z^\prime) ).
\]
Furthermore, \Cref{lm:sectZ(f)cor} tells us that $\Div^A(\delta\otimes s^\prime/d^n)^{\varpi^*(\mu^\prime),v^\pprime}_{\Delta^\pprime}=i^\prime\circ \langle \lambda'\rangle$
for some $\lambda'\in k[U^\prime]^\times$.
Let $\omega\in k[U]^\times$ be an invertible function on $U$ satisfying $\pi^*(\omega)(z)=\lambda'(z)^{-1}$. 
Define
$\Phi\defeq \Phi^\prime\circ \ip{\omega}$ and $\Theta\defeq \Theta^\prime\circ \ip{\pi^*(\omega)}$.
Then 
$\Theta\circ i_1= i^\prime\circ \ip{\lambda'\cdot\pi^*(\omega)}$ and so \Cref{lm:locup} yields the claim. 
\end{proof}

\begin{proof}[Proof of \Cref{th:etex}]
Lemmas \ref{lm:injACoretex} and \ref{lemma:etex-surj} establish respectively injectivity and surjectivity of the map $\pi^*$. 
\end{proof}

\subsubsection{}We finish this section with a result on the interplay between Zariski excision, étale excision and homotopy invariance for the cohomology theory $A^*$.

\begin{corollary}\label{cor:hty-inv-exc}
Suppose that $A^*$ is a graded presheaf of abelian groups that satisfies all properties of  a good cohomology theory except the étale excision axiom. Instead, assume that $A^*$ satisfies Zariski excision and homotopy invariance. In other words,
for any $X\in \Sm_k$, any line bundle $\scrL$ on $X$, any open subscheme $j\colon U\subseteq X$ and any closed subscheme $Z\subseteq X$ such that $Z\subseteq U$, the maps
\begin{align*}
&\pr^*\colon A^n(X,\scrL)\xrightarrow{\cong} A^n(X\times\A^1,\pr^*\scrL),\\
&j^*\colon A^n(X,X\setminus Z,\scrL)\xrightarrow{\cong} A^n(U,U\setminus Z,j^*\scrL)
\end{align*}
are isomorphisms.

Then $A^*$ satisfies the étale excision axiom on local schemes. In other words, for any $X\in\Sm_k$, $Z\subseteq X$, $\pi\colon (X^\prime, Z^\prime)\to (X,Z)$,
$z\in Z$ and $z^\prime\in Z^\prime$ as in \Cref{th:etex}, the morphism $\pi$ induces an isomorphism 
\[
\pi^*\colon A^*(X_z,X_z\setminus Z_z)\xrightarrow{\cong} A^* (X_{z'}^\prime,X_{z'}^\prime\setminus Z_{z'}^\prime).
\] 
\end{corollary}
\begin{proof}
Consider the category $\Cor^A_k$ of correspondences built from $A^*$ in the sense of \Cref{def:ACor}.
First of all we see that the proofs of Lemmas \ref{lm:injACoretex} and \ref{lemma:etex-surj} (as well as \Cref{constr:lrangleclass}) do not use the étale excision axiom for $A^*$.
Thus we have morphisms $\Phi_l, \Phi_r\in \Cor^A_k(U,X^\prime)$ such that $\pi\circ \Phi_r=i$, and $\Phi_r\circ\pi=i^\prime$.
Then $\Phi_l$ induces a right inverse $ A^* (X_{z'}^\prime,X_{z'}^\prime\setminus Z_{z'}^\prime)\to A^*(X_z,X_z\setminus Z_z)$ to $\pi^*$, and $\Phi_r$ induces a left inverse.
\end{proof}

\section{The cancellation theorem}\label{section:cancel}
In this section we show the cancellation theorem for $A$-correspondences by suitably adapting Voevodsky's proof for the case of $\Cor_k$ \cite{Voe-cancel}; see \Cref{thm:cancel}. For the sake of brevity we will omit the steps that are identical to Voevodsky's original proof, and rather focus on the details that are specific to our situation. We refer the interested reader to \cite{Voe-cancel} for the remaining formal aspects of the proof.

\begin{definition}
The \emph{Karoubi envelope} of $\Cor^A_k$ is the preadditive category whose objects are pairs $(X,p)$ with $X\in\Sm_k$ and $p\in\Cor^A_k(X,X)$ an idempotent. The morphisms are given by 
\[
\Cor^A_k((X,p), (X^\prime,p^\prime)) = \im\p*{\Cor^A_k(X,X^\prime)\xrightarrow{p'\circ (-)\circ p}\Cor^A_k(X,X^\prime)}.
\]
Any object $X\in \Sm_k$ can be considered as an object of the Karoubi envelope of $\Cor^A_k$ by $X\mapsto(X,\id_X)$. By abuse of notation, we will write $\Cor^A_k$ also for the Karoubi envelope of $\Cor^A_k$.
\end{definition}
\newcommand{\Gmw}{\G_m^{\wedge 1}}
\newcommand{\inj}{\iota}
\begin{definition}\label{def:wedgeGm}
Define
$X\wedge \Gmw \defeq \ker(\pr_1\colon X\times\G_m\to X)$ as an object of the Karoubi envelope of $\Cor^A_k$. 
Let $\pr^\wedge\colon \G_m^{\times2}\to \G_m^{\wedge 2}$ denote the canonical projection,
and let $\inj^\wedge\colon \G_m^{\wedge 2}\to \G_m^{\times2}$ denote the canonical injection. Note that $\pr^\wedge \circ \inj^\wedge=\id_{\G_m^{\wedge 2}}$.
The external product on $A$-correspondences defines a functor 
$
(-)\wedge \G^{\wedge 1}_m\colon \Cor^A_k\to \Cor^A_k
$
given by $X\to X\wedge \Gmw$, $\alpha\mapsto \alpha\times \id_{\G^{\wedge 1}_m}$. 
Furthermore, for any $X\in\Sm_k$ we let $\rmc_A(X)\wedge\Gmw$ denote the presheaf $U\mapsto\Cor^A_k(U\wedge\G_m,X\wedge\G_m)$.
\end{definition}

\begin{lemma}\label{lm:twistGm2}
Let $\tau^\times\colon \G_m^{\times2}\to \G_m^{\times 2}$ denote the twist automorphism given by $\tau(x_1,x_2)\defeq(x_2,x_1)$, and let 
\[
\tau^\wedge\defeq \pr^\wedge\circ\tau^\times\circ\inj^\wedge\colon\G_m^{\wedge2}\to\G_m^{\wedge2}.
\]
Then $\tau^\wedge$ is $\A^1$-homotopic to $\epsilon=-\ip{-1}\in\Cor^A_k(\G^{\wedge 2}_m, \G^{\wedge 2}_m)$.
\end{lemma}

\begin{proof}
Let $(x_1,x_2)$ denote the coordinates on $\G^{\times2}_m$. Denote by $\Delta\subseteq\G_m^{\times 2}$ the diagonal, and by $\wh\Delta\subseteq\G_m^{\times 2}$ the anti-diagonal, i.e.,
\[
\Delta\defeq Z(x_1x_2^{-1}-1),\quad \wh\Delta\defeq Z(x_1x_2-1)\subseteq\G_m^{\times2}.
\]
Let us first show that $\pr^\wedge\circ \tau^\times \circ j\sim_{\A^1} \epsilon\circ j$, where 
$
j\colon \G_m^{\times2}\setminus (\Delta\cup \wh \Delta)\to \G_m^{\times 2}
$
denotes the inclusion and $\epsilon=-\ip{-1}\in \Cor^A_k(\G^{\times 2}_m, \G^{\times 2}_m)$. 
\newcommand{\Gmph}{\G_m^{\phantom{1}}}
To do this, consider the diagram 
\[\begin{tikzcd}
\stackrel{t}{\Gmph}\times((\stackrel{x_1}{\Gmph}\times\stackrel{x_2}{\Gmph})\setminus (\Delta\cup \wh \Delta))\ar{r}{f}\ar{d}[swap]{p}\ar{dr}{g} & \A^1\\
(\stackrel{x_1}\Gmph\times\stackrel{x_2}{\Gmph})\setminus (\Delta\cup \wh \Delta) & \G_m\times\G_m
\end{tikzcd}\]
in which $g(t,x_1,x_2)\defeq (t, x_1x_2 t^{-1})$.
Then $p$ is a smooth relative curve whose relative canonical class is trivialized by $dt$. Applying \Cref{constr:lrangleclass} to this diagram we obtain a finite $A$-correspondence
\[
\Div^A(f)_Z^{dt,g}\in \Cor^A_k(\G_m^{\times 2}\setminus (\Delta\cup \wh \Delta), \G^{\times 2}_m)
\]
for any regular function $f$ whose vanishing locus $Z$ is finite over $\G_m^{\times2}\setminus (\Delta\cup \wh \Delta)$.
For simplicity, let us skip $dt$ and $g$ in the notation.
Then the required $\A^1$-homotopy is given as follows:
\begin{align*}
&(\tau^\times + \langle -1\rangle)\circ j\\ 
&= \p*{\Div^A( (t-x_1)(t-x_2))_{Z(t-x_2)} +  \Div^A( (t-x_1)(t-x_2))_{Z(t-x_1)} }\circ j \circ \langle (x_2-x_1)^{-1} \rangle   \\
&= \Div^A( (t-x_1)(t-x_2))\circ j \circ \langle (x_2-x_1)^{-1} \rangle\\
&\sim_{\A^1} \Div^A( (t-x_1 x_2)(t-1)) \circ j\circ \langle (x_2-x_1)^{-1} \rangle  \\
&= \p*{\Div^A( (t-x_1 x_2)(t-1))_{Z(t-1)} + \Div^A( (t-x_1 x_2)(t-1))_{Z(t-x_1 x_2)}}\circ j \circ \langle (x_2-x_1)^{-1} \rangle \\
&= ( \nu_1 + \nu_2 )\circ i \circ \langle (1-x_1x_2) (x_2-x_1)^{-1} \rangle\in \Cor^A_k(\G_m^{\times 2}\setminus (\Delta\cup \wh \Delta), \G^{\times 2}_m).
\end{align*}
Here $\nu_1\colon \G_m^{\times 2}\to \G_m^{\times 2}$ is the morphism $(x_1,x_2)\mapsto (x_1 x_2,1)$, while $\nu_2\colon \G^{\times 2}_m\to \G_m^{\times 2}$ is defined by $(x_1,x_2)\mapsto (1,x_1 x_2)$. Since $\pr^{\wedge} \circ \nu_1 = 0$ and $\pr^{\wedge} \circ \nu_2 = 0$ in $\Cor^A_k(\G_m^{\times 2}\setminus (\Delta\cup \wh \Delta), \G^{\wedge 2}_m)$, it follows that 
\[
\pr^\wedge\circ (\tau^\times + \langle -1\rangle)\circ j=0\in \ol\Cor^A_k(\G_m^{\times 2}\setminus (\Delta\cup \wh \Delta), \G^{\wedge 2}_m).
\]
Now \Cref{cor:Gm-diag} yields that 
\begin{equation}\label{eq:tautwis}
\pr^\wedge\circ (\tau^\times + \langle -1\rangle)=0\in \ol\Cor^A_k(\G_m^{\times 2}, \G^{\wedge 2}_m),
\end{equation}
since $\overline\Cor^A_k(-,\G_m^{\wedge 2})$ is a homotopy invariant presheaf with $A$-transfers.
Finally, since 
\[
\epsilon=-\pr^\wedge\circ\ip{-1}\circ\inj^\wedge\in\Cor_k^A(\G_m^{\wedge2},\G_m^{\wedge2}),
\]
we get the claim upon composing \eqref{eq:tautwis} with $\inj^\wedge$.
\end{proof}

\begin{definition}\label{def:rat-fu}
Let $\G_m\times\G_m$ have coordinates $(t_1,t_2)$. For any $n\ge1$, define the functions $g_n^+,g_n^-\in k[\G_m\times\G_m]$ by
\[
g_n^+\defeq t_1^n+1,\quad g_n^-\defeq t_1^n+t_2.
\]
Moreover, let $Z_n^\pm$ denote the support of the principal divisor $Z(g_n^\pm)$ on $\G_m\times\G_m$ defined by $g_n^\pm$.
\end{definition}

\begin{remark}\label{rmk:gn}
The functions $g_n^+/g_n^-$ differ by a sign from Voevodsky's functions $g_n$ defined in \cite[§4]{Voe-cancel}. However, the same proof as that of \cite[Lemma 4.1]{Voe-cancel} goes through to show that for any closed subset $T$ of $\G_m\times X\times \G_m\times Y$ finite and surjective over $\G_m\times X$, there is an integer $N$ such that for all $n\ge N$, the divisor of $g^+_n/g^-_n$ intersects $T$ properly over $X$, and the associated cycle is finite over $X$. The only reason for our choice of functions is to make the finite $A$-correspondence in \Cref{lemma:homotopy} homotopic to $\ip1$, and not $\ip{-1}$. Of course, in the situation of \cite{Voe-cancel} this choice does not matter, as Voevodsky's correspondences are oriented.
\end{remark}

\begin{definition}
Let $Y\in\Sm_k$, and recall from \Cref{def:wedgeGm} the definition of the presheaf $\rmc_A(Y)\wedge \Gmw$. Given any integer $n\ge1$, we will construct maps of presheaves
\[\begin{tikzcd}
\rmc_A(Y)\ar[r,shift left,"\theta"] & \rmc_A(Y)\wedge \Gmw\ar[l,shift left,"\rho_n"]
\end{tikzcd}\]
as follows. 

Let $X\in\Sm_k$, and let $T$ be any admissible subset of $X\times Y$. Then the homomorphism
\[
\theta\colon A^{\dim Y}_T(X\times Y,\omega_Y)\to A^{\dim Y+1}_{T\times\Delta(\G_m)}(X\times\G_m\times Y\times\G_m,\omega_{Y\times\G_m})
\]
is defined by 
\[
\theta\defeq(-)\times\id_{\G_m}=(-)\times\Delta_*(1),
\] 
where $\Delta\colon\G_m\to\G_m\times\G_m$ is the diagonal. 
Since for any admissible $T$ in $X\times Y$ the subset $T\times\Delta(\G_m)$ is admissible in $X\times\G_m\times Y\times\G_m$,
the map $\theta$ is well defined. It follows that $\theta$ induces a map of presheaves
$
\theta\colon \rmc_A(Y)\to\rmc_A(Y)\wedge\Gmw.
$
On the other hand, the map
\[
\rho_n\colon A_{T}^{\dim Y+1}(X\times\G_m\times Y\times\G_m,\omega_{Y\times\G_m})\to A_{T\cap (Z_n^+\cup Z_n^-)}^{\dim Y}(X\times Y,\omega_Y)
\]
is defined in the following way. By applying \Cref{constr:lrangleclass} to the diagram
\[\begin{tikzcd}
\stackrel{t_1}{\G_m}\times\stackrel{t_2}{\G_m}\ar{r}{g_n^\pm}\ar{d}[swap]{\pr_2}\ar{dr}{\pr_1} & \A^1\\
\stackrel{t_2}{\G_m} & \stackrel{t_1}{\G_m}
\end{tikzcd}\]
we obtain finite $A$-correspondences $\Div^A(g_n^\pm)\in\Cor_k^A(\G_m,\G_m)$. We then define $\rho_n$ by the formula
\[\rho_n\defeq p_*\p*{(-)\smile q^*\p*{\Div^A(g_n^+)-\Div^A(g_n^-)}},\]
where $p$ and $q$ are the projections 
\[
p\colon X\times \G_m\times Y\times\G_m\to X\times Y,\quad q\colon X\times\G_m\times Y\times\G_m\to\G_m\times \G_m.
\] 
Thus $\rho_n$ is defined whenever the subset $T\cap (Z_n^+\cup Z_n^-)$ is admissible in $X\times Y$.
Now, note that for any 
$f\colon X^\prime\to X$ and $\Phi\in A_{T}^{\dim Y+1}(X\times\G_m\times Y\times\G_m,\omega_{Y\times\G_m})$, 
the element 
$\rho_n(f^*(\Phi))$ is defined whenever $\rho_n(\Phi)$ is defined, and $\rho_n(f^*(\Phi))=f^*(\rho_n(\Phi))$.
Secondly, 
for any $\Phi,\Psi\in A_{T}^{\dim Y+1}(X\times\G_m\times Y\times\G_m,\omega_{Y\times\G_m})$
the element $\rho_n(\Phi+\Psi)$ is defined whenever $\rho_n(\Phi)$ and $\rho_n(\Psi)$ are defined and $\rho_n(\Phi+\Psi)=\rho_n(\Phi)+\rho_n(\Psi)$. In this regard we refer to $\rho_n$ a \emph{partially defined map of presheaves}.
\end{definition}

\subsubsection{}The maps $\rho_n$ form an exhausting sequence of partially defined homomorphisms in the sense that for any finite subset $F\subseteq \Cor^A_k(X\wedge\G_m, Y\wedge\G_m)$, there is an integer $N(F)$ such that for all $n\ge N(F)$, $\rho_n(\alpha)$ is defined for all $\alpha\in F$. Indeed, this condition is satisfied by \Cref{rmk:gn}.

\begin{lemma}\label{lemma:homotopy}
Let $q'\colon\G_m\times\G_m\to\Spec k$ denote the projection, and let $\Delta\colon\G_m\to\G_m\times\G_m$ be the diagonal. Then there is an $\A^1$-homotopy 
\[q'_*\p*{ \Delta_*\p*{\Div^A( \Delta^*(g_n^+)) - \Div^A( \Delta^*(g^-_n))}}\sim_{\A^1}\ip1\in A^0(\Spec k,\calO_{\Spec k}).\]
\end{lemma}

\begin{proof}
We deduce the claim from the following computation:
\begin{align}
&q'_*\p*{ \Delta_*\p*{\Div^A( \Delta^*(g_n^+)) - \Div^A( \Delta^*(g^-_n))}}\\
&=\Div^A( \Delta^*(g_n^+))^{\pr^{\G_m}_{\pt}} - \Div^A( \Delta^*(g^-_n))^{\pr^{\G_m}_{\pt}}\label{lem:eq1}\\
&=\Div^A( \Delta^*(g_n^+))^{\pr^{\A^1}_{\pt}} - \Div^A( \Delta^*(g^-_n))_{Z(g^-_n|_{\G_m})}^{\pr^{\A^1}_{\pt}}\label{lem:eq2}\\
&=\Div^A( t^n+1)^{\pr^{\A^1}_{\pt}} - \Div^A( t^n+t)_{Z(t^{n-1}+1)}^{\pr^{\A^1}_{\pt}}\label{lem:eq3}\\
&\sim_{\A^1}\Div^A( t^n+ t)_{Z(t^n+ t)}^{\pr^{\A^1}_{\pt}} - \Div^A( t^n+t)_{Z(t^{n-1}+1)}^{\pr^{\A^1}_{\pt}}\label{lem:eq4}\\
&=\Div^A( t^n+ t)_{Z(t)}^{\pr^{\A^1}_{\pt}}=
\langle 1 \rangle.\label{lem:eq5}
\end{align}
Here the homotopy \eqref{lem:eq4} is given by $t^n+\lambda t+(1-\lambda)\in k[\A^1\times\A^1]$.
\end{proof}

\subsubsection{}We are now ready to prove the cancellation theorem for $A$-correspondences.

\begin{theorem}\label{thm:cancel}
For any $X,Y\in \Sm_k$, the map $\theta=(-)\wedge {\G^{\wedge 1}_m}$ 
induces a quasi-isomorphism of complexes of presheaves with $A$-transfers
\[
\rmC_*(\theta)\colon \Cor^A_k(\Delta^\bullet\times X,Y)\simeq \Cor^A_k((\Delta^\bullet\times X)\wedge \Gmw, Y\wedge \Gmw). 
\]
Here $\Delta^\bullet$ denotes the standard cosimplicial scheme over $k$, whose $n$-simplices $\Delta^n$ are given by $\Spec k[x_0,\dots,x_n]/(\sum_i x_i-1)$.
\end{theorem}

\begin{proof}
The proof follows the same approach as Voevodsky's cancellation theorem for the category $\Cor_k$ \cite{Voe-cancel}. Thus many aspects of the proof will be the same as those of Voevodsky's proof, and we will therefore focus on the details that are specific to our context.

To prove that $\rmC_*(\theta)$ is a quasi-isomorphism it is enough to show that the maps $\rho_n$ and $\theta$ are inverse to each other up to natural $\A^1$-homotopy. To this end, first note that the functions $g_n^+$ and $g_n^-$ enjoy the following properties:
\begin{itemize}
\item[(1)] $g_n^+\big|_{\Delta} = t^n +a_1 t^{n-1}+\dots +a_{n-1}t+1$, and $g_n^-\big|_{\Delta} = t^n+b_1 t^{n-1}+\dots+ b_{n-2} t^2 +t$ (in fact, $g_n^+\big|_{\Delta} = t^n+1$ and $g_n^-\big|_{\Delta} = t^n+t$);
\item[(2)] $g_n^+\big|_{\G_m\times 1}=g_n^-\big|_{\G_m\times 1}\neq 0$.
\end{itemize}
Let $p$ and $q$ be the projections 
\[
p\colon X\times \G_m\times Y\times\G_m\to X\times Y,\quad q\colon X\times\G_m\times Y\times\G_m\to\G_m\times \G_m.
\] 
Moreover, denote by $p'\colon X\times Y\to \Spec k$ and $q'\colon\G_m\times\G_m\to\Spec k$ the structure maps. Thus we have a pullback square
\[
\begin{tikzcd}
X\times\G_m\times Y\times \G_m\ar{r}{q}\ar{d}[swap]{p} & \G_m\times\G_m\ar{d}{q'}\\
X\times Y\ar{r}{p'} & \Spec k.
\end{tikzcd}
\]
Property (1) along with \Cref{lemma:homotopy} then implies that the composition $\rho_n\circ\theta$ is $\A^1$-homotopic to the identity, by the following computation:
\begin{align}
&p_*\p*{ (\alpha\times \Delta_*(1) ) \smile  q^*\p*{\Div^A( g_n^+) - \Div^A( g^-_n)}}\\
&=p_*\p*{
p^*(\alpha)\smile q^*\p*{\Delta_*(1)  \smile \p*{\Div^A( g_n^+) - \Div^A( g^-_n)}}}\label{cancel:eq1}\\
&= \alpha\smile p_*\p*{
q^*\p*{ \Delta_*(1)  \smile \p*{\Div^A( g_n^+) - \Div^A( g^-_n)}}}\label{cancel:eq2}\\
&= \alpha\smile (p')^*\p*{
q'_*\p*{\Delta_*(1)  \smile \p*{\Div^A( g_n^+) - \Div^A( g^-_n)}}}\label{cancel:eq3}\\
&= \alpha\smile (p')^*\p*{ 
q'_*\p*{ \Delta_*\p*{\Div^A(\Delta^*(g_n^+)) - \Div^A(\Delta^*(g^-_n))}}}\label{cancel:eq4}\\
&\sim_{\A^1}  \alpha\smile (p')^*(\ip1)\label{cancel:eq5}\\
&= \alpha.\label{cancel:eq}
\end{align}
Here the equality \eqref{cancel:eq2} follows from the projection formula, \eqref{cancel:eq3} follows from base change applied to the diagram above, and the homotopy \eqref{cancel:eq5} is given by \Cref{lemma:homotopy}.

Similarly, property (2) implies that for any $\alpha\in \Cor^A_k(X,Y)$, the
classes 
$\rho_n( (\alpha\times \id_{\G_m})\circ i_X)$, $\rho_n(i_ Y \circ (\alpha\times \id_{\G_m})\circ i_X)$ and $\rho_n(i_ Y \circ  (\alpha\times \id_{\G_m}))$ are equal to 0 up to natural homotopy, where $i_X\colon X\to X\times \G_m$ and $i_Y\colon Y\to Y\times \G_m$ denote the morphisms given by the rational point $1\colon \Spec k\to \G_m$.
Thus we see that $\rho_n\circ \theta\sim_{\A^1} \id_{\rmc_A(Y)}$.

Finally, \Cref{lm:twistGm2} implies that $\rho_n$ is also right inverse up to $\A^1$-homotopy by the same argument as \cite[Theorem 4.6]{Voe-cancel} (see also \cite[Lemma 7.5]{framed-cancel}).
\end{proof}

\section{The category of \texorpdfstring{$A$}{A}-motives}\label{section:motives}

In this section we assume that the base field $k$ is infinite, perfect and of characteristic different from $2$.

\subsection{Nisnevich localization}

\begin{theorem}\label{thm:ext}
The category of Nisnevich sheaves with $A$-transfers is abelian.
The Nisnevich sheafification $\scrF_{\Nis}$ of any presheaf with $A$-transfers $\scrF$ is equipped with $A$-transfers in a unique and natural way, and there is a natural isomorphism
\[
\Ext^i_{\Shv_{\Nis}(\Cor^A_k;\Z)}(\Z_A(X), \scrF_{\Nis} ) \cong \rmH^i_{\Nis}(X,\scrF_{\Nis}) .
\]
\end{theorem}

\begin{proof}
By \cite[Theorem 3.1]{DruDMGWeff} it is enough to show that
$
\Cor^A_k(U,X)
\cong 
\bigoplus_{x\in X} \Cor^A_k(U,X^h_x),
$
where $x\in X$ ranges over the set of all (not necessary closed) points.
Let $d_X$ denote the dimension of $X$. Then we have
\begin{align*}
\Cor^A_k(U,X) &=
\varinjlim\limits_{T\in \calA_0(U\times X/U)} A^{d_X}_T(U\times X,\omega_{X}) &\\
&\cong
\varinjlim\limits_{T\in \calA_0(U\times X/U)} 
\bigoplus_{x\in X} A^{d_X}_{T^h_{x}}(U\times X^h_x,\omega_{X^h_x})&\\
&=
\bigoplus_{x\in X} 
\varinjlim\limits_{T\in \calA_0(U\times X^h_x/U)} A^{d_X}_{T}(U\times X^h_x,\omega_{X^h_x})
\cong
\bigoplus_{x\in X} 
\Cor^A_k(U,X^h_x), 
\end{align*}
where the isomorphism in the second row is given by \Cref{lm:AEssSmSum}, and the isomorphism in the last row follows from \Cref{lm:EssSmACor=limA_T}.
\end{proof}

\begin{remark}
The category of finite $A$-correspondences $\Cor^A_k$ is a strict V-category of correspondences in the sense of \cite[Definition 2.3]{Garkusha-reconst},
and a V-ringoid in the sense of \cite[Definition 2.4]{K-motives}. So, alternatively,
\Cref{thm:ext} can be proved by using the technique of \cite{K-motives}.
Note also that the proof of \Cref{thm:ext} could be obtained by following the original approach of Suslin and Voevodsky \cite{Voe-hty-inv}, that is, showing that the cone of the morphism $\rmc_A(\calU^\bullet)\to \rmc_A(U)$ is acyclic. Here $\rmc_A(\calU^\bullet)$ is the \v{C}ech complex associated to a Nisnevich covering 
$\calU\to U$ of a smooth $k$-scheme $U$. 
\end{remark}

\subsection{Strict homotopy invariance}

\begin{theorem}\label{thm:strict-hty-inv}
Let $\scrF\in\PSh_\Sigma(\Cor^A_k;\Z)$ be a homotopy invariant presheaf with $A$-transfers. Then the associated Nisnevich sheaf $\scrF_{\Nis}$ is strictly homotopy invariant, i.e., the projection $p\colon X\times\A^1\to X$ induces an isomorphism
\[
p^*\colon \rmH^n_{\Nis}(X,\scrF_{\Nis})\xrightarrow{\cong} \rmH^n_{\Nis}(X\times\A^1,\scrF_{\Nis})
\]for all $X\in\Sm_k$ and all $n\ge0$. 
\end{theorem}

\begin{proof}
The theorem is a consequence of the injectivity and excision theorems proved in Sections \ref{section:inj-aff-line}, \ref{section:rel-aff-line}, \ref{section:inj-loc-sch} and \ref{section:et-exc}. The deduction of strict homotopy invariance from these results is formal; see for example \cite{hty-inv} or \cite{GWStrHomInv}.
\end{proof}

\subsection{Effective \texorpdfstring{$A$}{A}-motives}

\begin{definition}
The $\infty$-category $\DM_A^{\eff}(k)$ of \emph{effective $A$-motives} is the localization of the derived category $\D^-(\Shv_\Nis(\Cor^A_k;\Z))$ with respect to the morphisms of the form $X\times\A^1\to X$. Let  
$\rmM_A^{\eff}\colon \Sm_k\to \DM_A^{\eff}(k)$ be the functor defined as the composition of the localization $\D^-(\Shv_\Nis(\Cor^A_k;\Z))\to \DM_A^{\eff}(k)$ with the functor $\Sm_k\to \D^-(\Shv_\Nis(\Cor^A_k;\Z))$ given by $X\mapsto \Z_A(X)[0]$. For any $X\in\Sm_k$, we refer to $\rmM_A^\eff(X)$ as the \emph{effective $A$-motive} of $X$. If $X=\Spec k$, we abbreviate $\rmM_A^\eff(\Spec k)$ to $\Z_A$. Finally, we define the Tate object $\Z_A(1)$ as 
\[\Z_A(1)\defeq\cofib(\Z_A\to \rmM_A^\eff(\G_m))[-1],\] 
where $\Z_A\to \rmM_A^\eff(\G_m)$ is the map induced by the rational point $1\colon\Spec k\to\G_m$.
\end{definition}

\subsubsection{}Note that there is a symmetric monoidal structure on $\DM_A^\eff(k)$ inherited from that on $\Shv_\Nis(\Cor^A_k;\Z)$, satisfying $\rmM_A^{\eff}(X)\otimes \rmM_A^{\eff}(Y)\simeq \rmM_A^{\eff}(X\times Y)$. The motive of a point, $\Z_A$, is then the unit for this monoidal structure. For any $n\ge1$, we can use the monoidal structure to define $\Z_A(n)\defeq\Z_A(1)^{\otimes n}$.

\begin{theorem}[\protect{cf. \cite[Theorem 14.11]{MVW}}]
The $\infty$-category $\DM_A^\eff(k)$ of effective $A$-motives is equivalent to the full subcategory of $\D^-(\Shv_\Nis(\Cor^A_k;\Z))$ spanned by \emph{motivic complexes}, i.e., complexes whose cohomology sheaves are strictly homotopy invariant.
\end{theorem}

\begin{theorem}[\protect{cf. \cite[Proposition 14.16]{MVW}}]
Let $X\in \Sm_k$, and let $\scrF^\bullet$ be a motivic complex. Then there is a natural isomorphism 
\[
[\rmM_A^{\eff}(X),\scrF^\bullet[i]]_{\D^-(\Shv_\Nis(\Cor^A_k;\Z))}\cong \HH^i_{\Nis}(X,\scrF^\bullet)
\]
for each $i\ge0$.
\end{theorem}

\subsection{The category of \texorpdfstring{$A$}{A}-motives}As in the classical case, we obtain the category $\DM_A(k)$ of $A$-motives via a stabilization process with respect to tensoring with the Tate object. 

\begin{definition}
The $\infty$-category $\DM_A(k)$ of \emph{$A$-motives} is obtained from $\DM_A^\eff(k)$ by $\otimes$-inverting $\Z_A(1)$. There is then a canonical functor $\Sigma^\infty\colon\DM^\eff_A(k)\to\DM_A(k)$, and we define the functor $\rmM_{A}\colon \Sm_k\to \DM_A(k)$ as the composition of $\rmM_A^{\eff}$ and $\Sigma^\infty$.
\end{definition}

\subsubsection{}It follows similarly as in \cite{MW-cplx} that $\DM_A(k)$ is a presentably symmetric monoidal stable $\infty$-category equipped with an adjunction
$
\Sigma^\infty:\DM_A^\eff(k)\rightleftarrows\DM_A(k):\Omega^\infty.
$

\subsubsection{}The following result is a consequence of the cancellation theorem for $A$-correspondences:

\begin{theorem}
The canonical functor $\Sigma^\infty\colon\DM^{\eff}_A(k)\to \DM_A(k)$ is fully faithful, and 
for any $X\in \Sm_k$ and any motivic complex $\scrF^\bullet \in \D^-(\Shv_\Nis(\Cor^A_k;\Z))$,
there is a natural isomorphism 
\[
[\rmM_{A}(X),\Sigma^\infty\scrF^\bullet]_{\DM_A(k)}\cong \HH^i_{\Nis}(X,\scrF^\bullet).
\]
\end{theorem}

\begin{definition}
Let $X\in\Sm_k$. For any pair of integers $p,q\in\Z$, we define the \emph{$A$-motivic cohomology of $X$ in bidegree $(p,q)$} as
$
\rmH_A^{p,q}(X,\Z)\defeq[\rmM_A(X),\Z_A(q)[p]]_{\DM_A(k)}.
$
\end{definition}

\subsubsection{}
The adjunction $\gamma^*_A:\PSh_\Sigma(\Sm_k)\rightleftarrows\PSh_\Sigma(\Cor^A_k;\Z):\gamma_*^A$ descends to an adjunction
\begin{equation}\label{eq:stable-adj}
\gamma^*_A:\SH(k)\rightleftarrows\DM_A(k):\gamma_*^A
\end{equation}
of stable $\infty$-categories, which allows us to compare $\DM_A(k)$ with the motivic stable homotopy category $\SH(k)$.

\begin{definition}
Denote by $\sspt\in\SH(k)$ the motivic sphere spectrum. In the adjunction \eqref{eq:stable-adj} above, let $\rmH\Z_A\in\SH(k)$ denote the Eilenberg--Mac Lane spectrum $\rmH\Z_A\defeq\gamma_*^A\gamma_A^*(\sspt)$.
\end{definition}

\begin{lemma}
The spectrum $\rmH\Z_A$ is an $\calE_\infty$-ring spectrum in $\SH(k)$.
\end{lemma}

\begin{proof}
As the right adjoint $\gamma_*^A$ is lax symmetric monoidal, it follows that it preserves $\calE_\infty$-algebras. Now the left adjoint $\gamma^*_A$ is symmetric monoidal, so $\gamma_A^*(\sspt)$ is the unit in $\DM_A(k)$ and hence an $\calE_\infty$-algebra. We conclude that $\rmH\Z_A=\gamma_*^A\gamma^*_A(\sspt)$ is an $\calE_\infty$-ring spectrum.
\end{proof}

\subsubsection{}
The cancellation theorem for $A$-correspondences implies that $\rmH\Z_A$ is an $\Omega_{\rmT}$-spectrum in $\SH(k)$ which represents $A$-motivic cohomology. More precisely, for any $X\in\Sm_k$ and any pair of integers $p,q$, there is a natural isomorphism
$
[\Sigma^\infty_{\rmT} X_+,\Sigma^{p,q}\rmH\Z_A]_{\SH(k)}\cong \rmH^{p,q}_A(X,\Z).
$

\subsubsection{}The combination of \Cref{lemma:corr-cat} and \cite[Theorem 5.2]{Classification} shows moreover that in the above adjunction \eqref{eq:stable-adj}, the right adjoint is monadic:

\begin{theorem}
Let $e$ denote the exponential characteristic of $k$. Then there is an equivalence of presentably symmetric monoidal stable $\infty$-categories
\[
\Mod_{\rmH\Z_A[1/e]}(\SH(k))\simeq\DM_A(k,\Z[1/e]),
\]
where $\Mod_{\rmH\Z_A[1/e]}(\SH(k))$ denotes motivic spectra equipped with an action from $\rmH\Z_A[1/e]$.
\end{theorem}

\begin{remark}
Recall that the category $\SH^\eff(k)$ of \emph{effective spectra} is the stable subcategory of $\SH(k)$ generated under colimits by $\PP^1$-suspension spectra of smooth $k$-schemes. 
We note that Bachmann and Fasel's effectivity criterion \cite[Theorem 4.4]{effectivity} applies in our setting, showing that the spectrum $\rmH\Z_A\in\SH(k)$ is effective. 
G. Garkusha and I. Panin communicated to us orally that they proved this result independently using the category $\Z\rmF_*(k)$ of linear framed correspondences.
\end{remark}

\appendix
\section{Geometric ingredients}\label{appendix}

In this section we summarize the geometric facts and constructions used in the text.
In particular, we formulate a version of Serre's theorem on the existence of sections satisfying relevant properties, which is used in the proofs in Sections \ref{section:rel-aff-line}, \ref{section:inj-loc-sch} and \ref{section:et-exc}. We then provide the construction of the relative curves used in Sections \ref{section:inj-loc-sch} and \ref{section:et-exc}. Finally, we formulate a few lemmas that imply the finiteness conditions on the vanishing loci of the functions constructed in Sections \ref{section:rel-aff-line}, \ref{section:inj-loc-sch} and \ref{section:et-exc}.

All schemes considered in this appendix are assumed to be noetherian and separated.

\begin{prop}\label{prop:ZarMainTh}
For any étale morphism $e\colon U\to Y$ there is a decomposition
$U\xrightarrow{u} X\xrightarrow{p} Y$ with $p\circ u = e$, in which $u$ is a dense open immersion and $p$ is finite.
\end{prop}
\begin{proof}
This follows Zariski's Main Theorem \cite[III Corollary 11.4]{Hartshorne}.
\end{proof}

\subsubsection{Serre's theorem}The following lemma is a consequence of \cite[III Theorem 5.2]{Hartshorne}, and is used in Sections \ref{section:rel-aff-line}, \ref{section:inj-loc-sch} and \ref{section:et-exc}. In the text we refer to this result simply as Serre's theorem.

\begin{lemma}[Serre]
\label{prop:CorofSerreTh}
Let $\calO(1)$ be an ample invertible sheaf on a scheme $X$, and $\scrL$ be an invertible sheaf on $X$.
Then there is, for any closed subscheme $Z\subseteq X$, an integer $N\in \Z$ such that
the restriction homomorphism 
$
\Gamma(X,\scrL(l))\to \Gamma(Z,\scrL(l))
$
is surjective for all $l\ge N$.
Here $\scrL(l)\defeq\scrL \otimes\calO(l)$.
\end{lemma}

\begin{example}[Chinese remainder theorem]\label{ex:ChineseRemTh}
Let $U$ be an affine scheme. Suppose that $Z\subseteq \A^1_U$ is a closed subscheme, and that $v\in\calO_Z$ is a regular function on $Z$. 
Then, for all large enough $n$ there is a monic polynomial $f\in \calO_U[t]=\calO_{\A^1_U}$ of degree $n$ such that $f\big|_{Z}=v$.
\end{example}

\subsubsection{Construction of relative curves}
We now formulate the construction of relative curves used in the proofs of the étale excision theorems. For the proof we refer to \cite[Lemma 3.7]{GWStrHomInv}. Before stating the result, let us first recall the notion of an étale neighborhood:


\begin{definition}\label{def:et-nbhd}
Let $X$ be a scheme and suppose that $Z\subseteq X$ is a closed subscheme. If $\pi\colon X^\prime\to X$ is an étale morphism and $Z^\prime\subseteq X^\prime$ is a closed subscheme such that $\pi$ induces an isomorphism $Z^\prime\xrightarrow{\cong} Z$,
then we say that $\pi\colon (X^\prime,Z^\prime)\to (X,Z)$ is an \emph{étale neighborhood of $Z$ in $X$}.
\end{definition}

\begin{lemma}[\protect{\cite[Lemma 3.7]{GWStrHomInv}}]
\label{lm:RelCurves}
Let $k$ be a field and let $X$ be a smooth $k$-scheme. Suppose we are given a closed subscheme $Z\subseteq X$ along with an étale neighborhood $\pi\colon (X^\prime, Z^\prime)\to (X,Z)$ of $Z$ in $X$. Let moreover $z\in Z$ and $z^\prime\in Z^\prime$ be closed points such that $\pi(z^\prime)=z$, and write $U\defeq X_z$ and $U^\prime\defeq X^\prime_{z^\prime}$ for the corresponding local schemes.
Then there is a commutative diagram
\begin{equation}\label{diag:relCurves}\begin{gathered}\xymatrix{
U^\prime\ar[d] & \ovcCpp\ar[l]_{p''}\ar[d]^{\ol\varpi^\prime} &\cCpp\ar[l]_{j''}\ar[r]^{v^{\prime\prime}}\ar[d]^{\varpi^\prime} & X^\prime\ar@{=}[d]\\
U\ar@{=}[d] & \ovcCp\ar[l]_{p'}\ar[d]^{\ol\varpi} & \cCp\ar[l]_{j'}\ar[r]^{v^\prime}\ar[d]^{\varpi} & X^\prime\ar[d]^{\pi}\\
U & \ovcC\ar[l]_{p} &\cC\ar[l]_{j}\ar[r]^{v}& X
}\end{gathered}\end{equation}
in $\Sm_k$, such that the following properties hold:
\begin{itemize}
\item[(1)]
$p$, $p^\prime$, $p^{\prime\prime}$ are relative projective curves;
$j$, $j^\prime$, $j^{\prime\prime}$ are open immersions; $\varpi$, $\varpi^\prime$ are étale; $\ol\varpi$, $\ol\varpi^\prime$ are finite; and
$p \circ j$, $p^\prime\circ j^\prime,p^{\prime\prime}\circ j^{\prime\prime}$ are smooth.
Moreover,
$\calC^{\prime\prime}=\calC^{\prime}\times_U U^\prime$;
$\overline\calC^{\prime\prime}=\overline\calC^\prime\times_U U^\prime$; and 
there are trivializations of the relative canonical classes $\mu\colon \calO_{\calC}\cong \omega_{\calC/U}$ and $\mu'\colon \calO_{\calC^\prime}\cong \omega_{\calC^\prime/U}$.
\item[(2)]
The schemes 
$\cZ\defeq v^{-1}(Z)$, $\cZp\defeq v'^{-1}(\Zpri)$ and $\cZpp\defeq{v^{\prime\prime}}^{-1}(\Zpri)$ are finite over 
$U$ and $\Upri$, respectively.
\item[(3)]
There are closed subschemes 
$\Delta_Z\subseteq \calZ$, $\Delta_Z^\prime\subseteq \calZ^\prime$ and $\Delta_Z^{\prime\prime}\subseteq \calZ^{\prime\prime}$
such that $p$, $p^\prime$ and $p^{\prime\prime}$ induce isomorphisms 
$w\colon \Delta_Z\cong Z^\prime\times_X U$, $w^\prime\colon \Delta_Z^\prime\cong Z\times_X U$ and $w^{\prime\prime}\colon \Delta_Z^{\prime\prime}\cong Z^\prime\times_{X^\prime} U^\prime$.
Moreover,
$v\big|_{\calZ}\circ w^{-1}= \pr^{Z\times_X U}_Z$, 
$v^{\prime}\big|_{\calZ^\prime}\circ {w^\prime}^{-1} = \pi\big|_{Z^\prime}\circ \pr^{Z^\prime\times_X U}_{Z^\prime}$, and 
$v^{\prime\prime}\big|_{\calZ^{\prime\prime}}\circ {w^{\prime\prime}}^{-1}= \pr^{Z^\prime\times_{X^\prime} U^\prime}_{Z^\prime}$.
\item[(4)]
There are closed subschemes
$\Delta\subseteq\calC$ and $\Delta^\prime\subseteq \calC^{\prime\prime}$
such that 
$\Delta\times_{U} Z=\Delta_Z$, $\Delta^\prime\times_{\Upri} \Zpri=\Delta_Z^{\prime\prime}$
and such that
$p$ and $p^{\prime\prime}$ induce isomorphisms $p\big|_\Delta\colon \Delta\cong U$ and $p^{\prime\prime}\big|_{\Delta^\prime}\colon\Delta^\prime\cong U^\prime$. Moreover,
the compositions $v\circ p\big|_{\Delta}^{-1}$ and $v\circ p^{\prime\prime}\big|_{\Delta'}^{-1}$
are equal to the canonical morphisms $U\to X$ and $U^\prime\to X^\prime$, respectively.
\item[(5)]
The schemes 
$D\defeq\ol\calC\setminus \calC$, $D^\prime\defeq{\ol\calC}^\prime \setminus \calC^\prime$ and $D^{\prime\prime}\defeq\ovcCpp \setminus \cCpp$
are finite over $U$ and $U^\prime$ respectively.
Furthermore,
$D^{\prime\prime}\cong {\overline\varpi^\prime}^{-1}(D^\prime)$, and
$D^\prime\supseteq \ol\varpi^{-1}(D)$.
\item[(6)]
There is an ample line bundle $\calO(1)$ on $\overline{\calC}$ and a section $d\in \Gamma(\overline{\calC},\calO(1))$ such that $Z(d)=D$.
\end{itemize} 
\end{lemma}

\subsubsection{Finiteness of vanishing loci}
The following lemmas are used to prove that the zero loci of the functions constructed in Sections \ref{section:rel-aff-line}, \ref{section:inj-loc-sch} and \ref{section:et-exc} are finite over the relevant schemes.

\begin{lemma}[\protect{\cite[Lemma 4.1]{GWStrHomInv}}]
\label{lm:clEnbSect}
Let $U$ be a local scheme, and let $x\in U$ denote the closed point. Suppose that the residue field $k\defeq k(x)$ is infinite.
Let 
\[
\xymatrix{\Dpri\ar@{^(->}[r]^i\ar[rd] & \ovcCp \ar[r]^\pi\ar[d]^{p^\prime}& \ovcC\ar[dl]^p \\ & U}
\]
be a commutative diagram such that
\begin{itemize}
\item $p^\prime$ and $p$ are projective morphisms of relative dimension one;
\item $i$ is a closed immersion, and
\item $\pi$ and $p^\prime\circ i$ are finite.
\end{itemize}
Suppose furthermore that we are given the following data:
\begin{itemize}
\item an ample line bundle $\calO(1)$ on $\ovcCp$;
\item a section $d\in \Gamma(\ovcCp,\calO(1))$ such that $Z(d)\subseteq \Dpri$;
\item an invertible section $s_\infty\in \Gamma(\Dpri,\calO(1))$;
\item a closed subscheme $\calZ\subseteq \ovcC$ satisfying
$\cZp\cap \Dpri=\varnothing$, where $\cZp\defeq\pi^{-1}(\calZ)\subseteq\ovcCp$;
\item a section $s_{\cZp}\in \Gamma(\cZp,\calO(1))$ such that 
$\pi$ induces an isomorphism
$Z(s_{\cZp})\cong \pi(Z(s_{\cZp}))$.
\end{itemize}
Then there is
an integer $L\in\Z$ such that for all $l\ge L$, there is a section $s\in \Gamma(\ovcCp,\calO(l))$ satisfying
\begin{enumerate}
\item[$(1)$]
$s\big|_{\Dpri}=s_\infty^l$,
$s\big|_{\cZp} = s_{\cZp} d^{l-1}$;
\item[$(2)$]
$\pi$ induces an isomorphism $Z(s)\cong \pi(Z(s))$. 
\end{enumerate}
\end{lemma}

\begin{lemma}\label{lm:FinitnessOfVanLoc}
Let $U$ be a scheme and suppose that $\ovcC\to U$ is a projective morphism of pure dimension one.
Let $\scrL$ be an ample line bundle on $\ovcC$.
Then, for any pair of sections $d,e\in \Gamma(\ovcC,\scrL)$ such that $Z(d)\cap Z(e)=\varnothing$, 
the vanishing loci $Z(e)$ and $Z(d)$ are finite over $U$.
\end{lemma}

\begin{proof}
We prove that $Z(e)$ is finite over $U$; the case of $Z(d)$ follows by symmetry.
Since $\ovcC$ is projective over $U$, the same holds also for the closed subscheme $Z(e)$.
As $\ovcC$ is of pure dimension one, it follows that 
$Z(e)$ is finite over $U$ 
unless $Z(e)$ contains at least one irreducible component $C$ of the fiber $\ovcC\times_U x$ for some point $x\in U$.
But since $\scrL$ is ample, $\scrL\big|_{C}$ is nontrivial and hence $Z(d\big|_{C})\neq\varnothing$. So $Z(e)$ cannot contain an irreducible component of the fiber $\ovcC\times_U x$.
\end{proof}

\printbibliography
\end{document}